\documentclass[11pt,a4paper,english]{article}

\newcommand{\beq}{\begin{equation}}
\newcommand{\eeq}{\end{equation}}
\newcommand{\bea}{\begin{eqnarray}}
\newcommand{\eea}{\end{eqnarray}}
\newcommand{\beas}{\begin{eqnarray*}}
\newcommand{\eeas}{\end{eqnarray*}}

\usepackage[latin1]{inputenc}
\usepackage{babel}
\usepackage[centertags]{amsmath}
\usepackage{amsfonts}
\usepackage{amssymb}
\usepackage{color}
\usepackage{amsthm}
\usepackage{newlfont}
\usepackage{graphicx}
\usepackage{dsfont}
\usepackage{geometry}
\usepackage{mathrsfs}
\usepackage{authblk}
\usepackage{verbatim}
\usepackage{hyperref}

\newtheorem{theorem}{Theorem}[section]
\newtheorem{lemma}[theorem]{Lemma}

\newtheorem{prop}[theorem]{Proposition}
\newtheorem{definition}[theorem]{Definition}
\newtheorem{remark}[theorem]{Remark}
\newtheorem{ass}[theorem]{Assumption}

\def\RR{\mathbb R}
\def\FF{\mathbb F}
\def\EE{\mathsf E}
\def\PP{\mathsf P}
\def\QQ{\mathsf Q}

\def\cF{{\cal F}}
\def\cA{{\cal A}}

\def\cS{{\cal S}}
\def\cI{{\cal I}}
\def\cL{{\cal L}}
\def\cJ{{\cal J}}

\def\cM{{\cal M}}

\def\XX{\widetilde{X}}
\def\WW{\widetilde{W}}

\makeatletter  
\@addtoreset{equation}{section}
\def\theequation{\arabic{section}.\arabic{equation}}
\makeatother  
\begin{document}

\title{\textbf{Stochastic nonzero-sum games: \\ a new connection between singular control \\ and optimal stopping}}

\author{Tiziano De Angelis\thanks{School of Mathematics, University of Leeds, Woodhouse Lane, Leeds LS2 9JT, United Kingdom; \texttt{t.deangelis@leeds.ac.uk}}\:\:\emph{and}\:\:Giorgio Ferrari\thanks{Center for Mathematical Economics, Bielefeld University, Universit\"atsstrasse 25, D-33615 Bielefeld, Germany; \texttt{giorgio.ferrari@uni-bielefeld.de}}}

\date{\today}
\maketitle

\textbf{Abstract.} In this paper we establish a new connection between a class of $2$-player nonzero-sum games of optimal stopping and certain $2$-player nonzero-sum games of singular control. We show that whenever a Nash equilibrium in the game of stopping is attained by hitting times at two separate boundaries, then such boundaries also trigger a Nash equilibrium in the game of singular control. Moreover a differential link between the players' value functions holds across the two games. 

\medskip


\smallskip

{\textbf{MSC2010 subject classification}}: 91A15, 91A05, 93E20, 91A55, 60G40, 60J60, 91B76.



\section{Introduction}
\label{Introduction}

Connections between some problems of singular stochastic control (SSC) and questions of optimal stopping (OS) are well known in control theory.
In 1966 Bather and Chernoff \cite{BatherChernoff67} studied the problem of controlling the motion of a spaceship which must reach a given target within a fixed period of time, and with minimal fuel consumption. This problem of aerospace engineering was modeled in \cite{BatherChernoff67} as a singular stochastic control problem, and an unexpected link with optimal stopping was observed. The value function of the control problem was indeed differentiable in the direction of the controlled state variable, and its derivative coincided with the value function of an optimal stopping problem. 

The result of Bather and Chernoff was obtained by using mostly tools from analysis. Later on, Karatzas \cite{K81,K83}, and Karatzas and Shreve \cite{KaratzasShreve84} employed fully probabilistic methods to perform a systematic study of the connection between SSC and OS for the so-called ``monotone follower problem''. The latter consists of tracking the motion of a stochastic process (a Brownian motion in \cite{K81}, \cite{K83}, \cite{KaratzasShreve84}) by a nondecreasing control process in order to maximise (minimise) a performance criterion which is concave (convex) in the control variable. Further, a link to optimal stopping was shown to hold also for monotone follower problems of finite-fuel type; i.e.\ where the total variation of the control (the fuel available to the controller) stays bounded (see \cite{ElKK88}, \cite{Karatzas85}, and also \cite{Bank} for dynamic stochastic finite-fuel). More recent works provided extensions of the above results to diffusive settings in \cite{BenthReikvam} and \cite{BoetiusKohlmann}, to Brownian two-dimensional problems with state constraints in \cite{BudhirajaRoss}, to It\^o-L\'evy dynamics under partial information in \cite{OksendalSulem12}, and to non-Markovian processes in \cite{BK97}.

It was soon realised that these kinds of connections could be established in wider generality with admissible controls which are of bounded variation as functions of time (rather than just monotone). Indeed, under suitable regularity assumptions (including convexity or concavity of the objective functional with respect to the control variable) the value function of a bounded variation control problem is differentiable in the direction of the controlled state variable, and its derivative equals the value function of a $2$-player zero-sum game of optimal stopping (Dynkin game). To the best of our knowledge, this link was noticed for the first time in \cite{Taksar85} in a problem of controlling a Brownian motion, and then generalised in \cite{Boetius} and \cite{KW}, and later on also in \cite{GT08} via optimal switching.

It is important to observe that despite their appearance in numerous settings, connections between SSC and OS are rather ``delicate'' and should not be given for granted, even for monotone follower problems with very simple diffusion processes. Indeed, counterexamples were recently found in \cite{DeAFeMo15b} and \cite{DeAFeMo15c} where the connection breaks down even if the cost function is arbitrarily smooth and the underlying processes are Ornstein-Uhlenbeck or Brownian motion. 

The existing theory on the connection between SSC and OS is well established for \emph{single agent} optimisation problems. However, the latter are not suitable for the description of more complex systems where strategic interactions between several decision makers play a role. Problems of this kind arise for instance in economics and finance when studying productive capacity expansion in an oligopoly \cite{Steg10}, the competition for the market-share control \cite{KZ15}, or the optimal control of an exchange rate by a central bank (see the introduction of the recent \cite{HHSZ} for such an application).

In this paper we establish a new connection between a class of $2$-player nonzero-sum games of optimal stopping (see \cite{DeAFeMo15} and references therein) and certain $2$-player nonzero-sum games of singular stochastic control. These games involve two different underlying (one-dimensional) It\^o-diffusions. The one featuring in the game of controls will be denoted by $\XX$, whereas the one featuring in the game of stopping will be denoted by $X$.

In the game of controls each player may exert a monotone control to adjust the trajectory of $\widetilde{X}$. The first player can only increase the value of $\widetilde{X}$, by exerting her control, while the second player can only decrease the value of $\XX$, by exerting her control. If player 1 uses a unit of control at time $t>0$, then she must pay $G_1(\widetilde{X}_t)$, while at the same time player 2 receives $L_2(\widetilde{X}_t)$. A symmetric situation occurs if player 2 exerts control (see Section \ref{sec:contr}). Each player wants to maximise her own total expected reward functional. 

In the game of stopping both players observe the dynamics of $X$ and may decide to end the game, by choosing a stopping time for $X$. When the game ends, each player pays a cost according to the following rule: if the $i$-th player stops first, she pays $G_i$; if instead the $i$-th player lets the opponent stop first, she pays $L_i$. Here $G_i$ and $L_i$ are the same functions as in the game of controls, and in general they depend on the value of $X$ at the random time when the game is ended.

We show that if a Nash equilibrium in the game of stopping is attained by hitting times of two separate thresholds, i.e.~the process $X$ is stopped as soon as it leaves an interval $(a_*,b_*)$ of the real line, then the couple of controls that keep $\widetilde{X}$ inside $[a_*,b_*]$ with minimal effort (i.e.\ according to a Skorokhod reflection policy) realises a Nash equilibrium in the game of singular controls. Moreover, we also prove that the value functions of the two players in the game of singular controls can be obtained by suitably integrating their respective ones in the game of optimal stopping. The existence of Nash equilibria of threshold type for the game of stopping holds in a large class of examples as it is demonstrated in the recent \cite{DeAFeMo15}. Here the proof of our main theorem (cf.\ Theorem \ref{thm:labomba} below) is based on a verification argument following an educated guess. In order to illustrate an application of our results we present a game of pollution control between a social planner and a firm representative of the productive sector.

Another important result of this paper is a simple explicit construction of Markov-perfect equilibria\footnote{i.e.~equilibria in which each player dynamically reacts to her opponent's decisions} for a class of 2-player continuous time stochastic games of singular control. This is a problem in game theory which has not been solved in full generality yet (see the discussion in Section 2 of \cite{BP09} and in \cite{Steg10}), and here we contribute to further improve results in that direction. We seek for Nash equilibria in the class of control strategies $\mathcal{M}$ which forbids the players to exert simultaneous impulsive controls (i.e.~simultaneous jumps of their control variables). On the one hand, this is a convenient choice for technical reasons, but, on the other hand, we also show in Appendix \ref{app-setM} that it induces no loss of generality in a large class of problems commonly addressed in the literature on singular stochastic control.

It is worth emphasising a key difficulty in handling nonzero-sum games. If, e.g., player $1$ deviates unilaterally from an equilibrium strategy this has two effects: it worsens player 1's performance, but it also affects player 2's payoff. However it is impossible to establish \emph{a priori} whether such a deviations benefit or harm player 2. This issue does not arise in single-agent problems and in two-player zero-sum games where the optimisation involves a unique objective functional. From a PDE point of view this is expressed by the fact that our nonzero-sum game of controls is associated to a system of coupled variational inequalities, rather than to a single variational inequality. Thus there is a fundamental difference between the nature of our results and the one of those already known for certain (single-agent) bounded variation control problems (see e.g.~\cite{Boetius}, \cite{Taksar85}). 

Our work marks a new step towards a global view on the connection between singular stochastic control problems and questions of optimal stopping by extending the existing results to nonzero-sum, multi-agent optimisation problems. A link between these two classes of optimisation problems is important not only from a purely theoretical point of view but also from a practical point of view. Indeed, as it was pointed out in \cite{KaratzasShreve84} (cf.~p.\ 857) one may hope to ``jump'' from one formulation to the other in order to ``pose and solve more favourable problems''. As an example, one may notice that questions of existence and uniqueness of optimisers are more tractable in control problems, than in stopping ones; on the other hand, a characterisation of optimal control strategies is in general a harder task than the one of optimal stopping rules. Recent contributions to the literature (e.g., \cite{CH09}, \cite{DeAFe14} and \cite{Ferrari}) have already highlighted how the combined approach of singular stochastic control and optimal stopping is extremely useful to deal with investment/consumption problems for a single representative agent. It is therefore reasonable to expect that our work will increase the mathematical tractability of investment/consumption problems for multiple interacting agents.

The rest of the paper is organised as follows. In Section \ref{sec:setting} we introduce the setting, the game of singular controls and the game of optimal stopping. In Section \ref{sec:main} we prove our main result and we discuss the assumptions needed. An application to a game of pollution control is considered in Section \ref{sec:pollution}, whereas some proofs and a discussion regarding admissible strategies are collected in the appendix.


\section{Setting}
\label{sec:setting}

\subsection{The underlying diffusions}
\label{sec:gc}

Denote by $(\Omega, \mathcal{F},\PP)$ a complete probability space equipped with a filtration $\FF=(\mathcal{F}_t)_{t \geq 0}$ under usual hypotheses. Let $\WW=(\WW_t)_{t\geq 0}$ be a one-dimensional standard Brownian motion adapted to $\FF$, and $(\XX^{\nu,\xi}_t)_{t\ge 0}$ the strong solution (if it exists) to the one-dimensional, controlled stochastic differential equation (SDE)  
\begin{equation}
\label{state:X}
d\XX^{\nu,\xi}_t=\mu(\XX^{\nu,\xi}_t)dt+\sigma(\XX^{\nu,\xi}_t)d\WW_t+d\nu_t-d\xi_t,\qquad \XX^{\nu,\xi}_0=x\in\cI,
\end{equation}
with $\cI:=(\underline{x},\overline{x})\subseteq \RR$ and with $\mu$, $\sigma$ real valued functions which we will specify below.
Here $(\nu_t)_{t\ge 0}$ and $(\xi_t)_{t\ge0}$ belong to 
\begin{align}\label{def:S}
\cS:=\big\{\eta: (\eta_t(\omega))_{t\ge 0}\: \text{left-continuous, adapted, increasing, with $\eta_0=0$, $\PP$-a.s.}\big\}
\end{align}
and we denote 
\begin{align}\label{eq:sigmaI}
\sigma_{\,\cI}:=\inf\{t\ge0\,:\,\XX^{\nu,\xi}_t\notin \cI\}
\end{align}
the first time the controlled process leaves $\cI$. 

Notice that $\nu$ and $\xi$ can be expressed as the sum of their continuous part and pure jump part, i.e.
\begin{align}
\nu_t=\nu^c_t+\sum_{s<t}\Delta \nu_s, \quad \xi_t=\xi^c_t+\sum_{s<t}\Delta \xi_s,
\end{align}
where $\Delta\nu_s:=\nu_{s+}-\nu_s$ and $\Delta\xi_s:=\xi_{s+}-\xi_s$. 
Throughout the paper we will consider the process $\XX^{\nu,\xi}$ killed at $\sigma_\cI$, and we make the following assumptions on $\mu$ and $\sigma$. 

\begin{ass}
\label{ass:coef}
The functions $\mu$ and $\sigma$ are in $C^1(\cI)$ and $\sigma(x)>0$, $x\in\cI$. 
\end{ass}
\noindent Because $\mu$ and $\sigma$ are locally Lipschitz, for any given $(\nu,\xi) \in \cS \times \cS$ equation \eqref{state:X} has a unique strong solution (Theorem V.7 in \cite{Protter} and the text after its proof).

To account for the dependence of $\XX$ on its initial position, from now on we shall write $\XX^{x,\nu,\xi}$ where appropriate. In the rest of the paper we use the notation $\EE_x[f(\XX^{\nu,\xi}_t)]=\EE[f(\XX^{x,\nu,\xi}_t)]$, for $f$ Borel-measurable, since $(\XX,\nu,\xi)$ is Markovian but the initial value of the controls is always zero. Here $\EE_x$ is the expectation under the measure $\PP_x(\,\cdot\,):=\PP(\,\cdot\,|\XX_0=x)$ on $(\Omega,\cF)$.
As mentioned in the introduction, \eqref{state:X} will be the underlying process in the game of control. 

To keep the notation simple and avoid introducing another filtered probability space, we also assume that the filtered probability space $(\Omega,\cF,\FF,\PP)$ is sufficiently rich to allow for the treble $(\Omega,\cF,\PP)$, $\FF$, $(X,W)$ to be a weak solution to the SDE 
\begin{align}\label{state:XX}
dX_t=\big(\mu(X_t)+\sigma(X_t)\sigma'(X_t)\big)dt+\sigma(X_t)dW_t,\quad X_0=x\in\cI,
\end{align}
where $W$ is another Brownian motion. Notice that this requirement does not affect generality of our results because $\XX$ and $X$ never feature at the same time in our optimisation problems. In particular $X$ will appear only in the game of stopping. 

Assumption \ref{ass:coef} guarantees that the above SDE admits a weak solution which is unique in law up to a possible explosion time \cite[Ch.~5.5]{KS}. Indeed for every $x \in \mathcal{I}$ there exists $\varepsilon_o>0$ such that
\begin{equation}
\label{LI}
\int_{x-\varepsilon_o}^{x+\varepsilon_o}\frac{1 + |\mu(z)|+|\sigma(z)\sigma'(z)|}{|\sigma(z)|^2}\,dz < +\infty.
\end{equation}
To account explicitly for the initial condition, we denote by $X^x$ the solution to \eqref{state:XX} starting from $x\in\cI$ at time zero. 
Due to \eqref{LI} the diffusion $X$ is regular in $\mathcal{I}$; that is, if $\tau_z:=\inf\{t\geq0: X^x_t=z\}$ one has $\PP(\tau_z<\infty) > 0$ for every $x$ and $z$ in $\mathcal{I}$ so that the state space cannot be decomposed into smaller sets from which $X$ cannot exit (see \cite[Ch.~2]{BS}). 

We make the following standing assumption. 
\begin{ass}
\label{ass:boundary}
The points $\underline{x}$ and $\overline{x}$ are either natural or entrance-not-exit for the diffusion $X$, hence unattainable. Moreover, $\underline{x}$ and $\overline{x}$ are unattainable for the uncontrolled process $\XX^{0,0}$.
\end{ass}
For boundary behaviours of diffusions one may consult p.~15 in \cite{BS}. Unattainability of $\underline x$ and $\overline x$ refers to the fact that, for $x\in\cI$, the processes $X^x$ and $\XX^{x,0,0}$ cannot leave the interval $(\underline x,\overline x)$ in finite time, $\PP$-a.s. Feller's test for explosion (see, e.g., Theorem 5.5.29 in \cite{KS}) provides necessary and sufficient conditions under which $\underline{x}$ and $\overline{x}$ are unattainable for the diffusions $X$ and $\XX^{0,0}$. Moreover, specific properties of natural and entrance-not-exit boundaries may be addressed by using the speed measure $m(dx)$ and the scale function $S(x)$ of the above diffusions (since we are not going to make use of these concepts we simply refer the interested reader to pp.\ 14--15 in \cite{BS} for details).

In the next remark we show that if $\sigma'$ is sufficient integrable, then unattainable boundary points of $X$ are also unattainable for the uncontrolled process $\XX^{0,0}$.
\begin{remark}
\label{rem:tilde}
For simplicity let us assume that $\sigma \in C^2(\mathcal{I})$ so that both \eqref{state:X} and \eqref{state:XX} admit strong solution. For $x\in\cI$ let us define a new measure $\QQ_x$ by the Radon-Nikodym derivative 
\begin{align*}
Z_t:=\frac{d\QQ_x}{d\PP_x}\bigg|_{\cF_t}=\exp\Big\{\int^t_0\sigma'(\XX^{0,0}_s)d\WW_s-\frac{1}{2}\int^t_0(\sigma')^2(\XX^{0,0}_s)ds\Big\},\qquad\PP_x-\text{a.s.}
\end{align*}
which is an exponential martingale under suitable integrability conditions on $\sigma'$. Hence Girsanov theorem implies that the process $B_t:=\WW_t-\int_0^t\sigma'(\XX^{0,0}_s)ds$ is a standard Brownian motion under $\QQ_x$ and it is not hard to verify that $\text{Law}\,(\XX^{0,0}\big|\QQ_x)=\text{Law}\,(X\big|\PP_x)$.  

It follows that denoting $\sigma^0_{\cI}=\inf\{t>0\,:\,\XX^{0,0}_t\notin\cI\}$ and $\tau_{\cI}=\inf\{t>0\,:\,X_t\notin\cI\}$ we have that $\text{Law}\,(\sigma^0_{\cI}|\QQ_x)=\text{Law}\,(\tau_{\cI}|\PP_x)$. Notice also that the measures $\QQ_x$ and $\PP_x$ are equivalent on $\cF^{\WW}_t$ for all $0\le t<+\infty$, where $(\cF^{\WW}_t)_{t\ge0}$ is the filtration generated by $\WW$ (see \cite{KS}, Chapter 3.5). In particular $\{\sigma^0_{\cI}\le t\}\in\cF^{\WW}_t$. Therefore, using that $\underline{x}$ and $\overline{x}$ are unattainable for $X$, we get
\begin{align*}
0=\PP_x(\tau_{\cI}\le t)=\QQ_x(\sigma^0_{\cI}\le t)\implies\PP_x(\sigma^0_{\cI}\le t)=0
\end{align*}
for all $t>0$. Hence, $\PP_x(\sigma^0_{\cI}<+\infty)=0$ which proves that $\underline{x}$ and $\overline{x}$ are unattainable for the process $\XX^{0,0}$ under $\PP_x$ for all $x\in\cI$.
\end{remark}

The infinitesimal generator of the uncontrolled diffusion $\XX^{x,0,0}$ is denoted by $\cL_{\XX}$ and is defined as
\begin{align}\label{eq:LXX}
(\cL_{\XX} f)\,(x):=\frac{1}{2}\sigma^2(x)f''(x)+\mu(x)f'(x),\quad f\in C^2(\overline{\cI}),\, x\in\cI,
\end{align}
whereas the one for $X$ is denoted by $\cL_{X}$ and is defined as
\begin{align}\label{eq:LX}
(\cL_{X} f)\,(x):=\frac{1}{2}\sigma^2(x)f''(x)+(\mu(x)+\sigma(x)\sigma'(x))f'(x),\quad f\in C^2(\overline{\cI}), x\in\cI.
\end{align}
Letting $r>0$ be a fixed constant, we assume
\begin{ass}\label{ass:rate}
$r>\mu'(x)$ for $x\in\overline{\cI}$.
\end{ass}
\noindent We denote by $\psi$ and $\phi$ the fundamental solutions of the ODE (see \cite[Ch.~2, Sec.~10]{BS})
\begin{align}\label{ODE}
\cL_X u(x)-(r-\mu'(x))u(x)=0,\qquad x\in\cI,
\end{align}
and we recall that they are strictly increasing and decreasing, respectively.

Finally, we denote by $S'(x)$, $x\in\cI$, the density of the scale function of $(X_t)_{t\ge 0}$, and by $w$ the Wronskian
\begin{equation}
\label{Wronskian}
w:= \frac{\psi'(x)\phi(x) - \phi'(x)\psi(x)}{S'(x)}, \quad x \in \cI,
\end{equation}
which is a positive constant.

Particular attention in this paper is devoted to solutions of \eqref{state:X} reflected inside intervals $[a,b]\subset{\cI}$, and we recall here the following result on Skorokhod reflection. Its proof can be found, for instance, in \cite[Thm.~4.1]{Ta79} (notice that $\mu'$ and $\sigma'$ are bounded on $[a,b]$).
\begin{lemma}\label{lem:Sk}
Let Assumption \ref{ass:coef} hold. For any $a,b\in\cI$ with $a< b$ and any $x\in [a,b]$ there exists a unique couple $(\nu^a,\xi^b)\in\cS\times\cS$ that solves the Skorokhod reflection problem $\textbf{SP}(a,b;x)$ defined as: 
\begin{align}
\tag{$\textbf{SP}(a,b;x)$}
\hspace{-10pt}
\text{Find $(\nu,\xi)\in\cS\times\cS$ s.t.}
\left\{
\begin{array}{l}
\XX^{x,\nu,\xi}_t\in[a,b], \text{$\PP$-a.s.~for $0< t\le\sigma_\cI $},\\[+6pt]
\int^{T\wedge\sigma_\cI}_0{\mathds{1}_{\{\XX^{x,\nu,\xi}_t>a\}}d\nu_t}=0, \text{$\PP$-a.s.~for any $T>0$,}\\[+6pt]
\int^{T\wedge\sigma_\cI}_0{\mathds{1}_{\{\XX^{x,\nu,\xi}_t<b\}}d\xi_t}=0, \text{$\PP$-a.s.~for any $T>0$.}
\end{array}
\right.
\end{align}
It also follows that $\text{supp}\{d\nu^a_t\}\cap \text{supp}\{d\xi^b_t\}=\emptyset$. 
\end{lemma}
For future frequent use we also recall the one-sided version of the above result.
\begin{lemma}
\label{lem:Sk01}
Let Assumption \ref{ass:coef} hold. For any $a\in\cI$, $x\geq a$ and $\xi\in\cS$ there exists a unique $\nu^a\in\cS$ that solves the Skorokhod reflection problem $\textbf{SP}^{\,\xi}_{a+}(x)$ defined by 
\begin{align}
\tag{$\textbf{SP}^\xi_{a+}(x)$}
\hspace{-10pt}
\text{find $\nu\in\cS$ s.t.}
\left\{
\begin{array}{l}
\XX^{x,\nu,\xi}_t\in[a,\overline{x}), \text{$\PP$-a.s.~for $0< t\le\sigma_\cI $},\\[+6pt]
\int^{T\wedge\sigma_\cI}_0{\mathds{1}_{\{\XX^{x,\nu,\xi}_t>a\}}d\nu_t}=0, \text{$\PP$-a.s.~for any $T>0$}.
\end{array}
\right.
\end{align}
Similarly, for any $b\in\cI$, $x\leq b$ and $\nu\in\cS$ there exists a unique $\xi^b\in\cS$ that solves the Skorokhod reflection problem $\textbf{SP}^{\,\nu}_{b-}(x)$ defined by
\begin{align}
\tag{$\textbf{SP}^\nu_{b-}(x)$}
\hspace{-10pt}
\text{find $\xi\in\cS$ s.t.}
\left\{
\begin{array}{l}
\XX^{x,\nu,\xi}_t\in(\underline{x},b], \text{$\PP$-a.s.~for $0< t\le\sigma_\cI $},\\[+6pt]
\int^{T\wedge\sigma_\cI}_0{\mathds{1}_{\{\XX^{x,\nu,\xi}_t<b\}}d\xi_t}=0, \text{$\PP$-a.s.~for any $T>0$.}
\end{array}
\right.
\end{align}
\end{lemma}
The proof of the above lemma is based on a Picard iteration scheme. Although this derivation seems to be standard we could not find a precise reference for our particular setting, and we provide a short proof in Appendix \ref{app-proofs}.

\subsection{The game of controls}\label{sec:contr}

We introduce a 2-player nonzero-sum game of singular control, where player 1 (resp.~player 2) can influence the dynamics \eqref{state:X} by exerting the control $\nu$ (resp.~$\xi$). The game has the following structure: if player 1 uses a unit of control at time $t>0$, she must pay a cost $G_1(\XX^{\nu,\xi}_t)$, while player 2 receives a reward $L_2(\XX^{\nu,\xi}_t)$. A symmetric situation occurs if player 2 exerts control. Both players want to maximise their own expected discounted reward functional $\Psi_i$ defined by 
\begin{align}
\label{eq:Psi1}\Psi_1(x;\nu,\xi):=\EE\Big[\int_0^{\sigma_\cI}e^{-rt}L_1(\XX^{x,\nu,\xi}_t)\,{\scriptstyle{\ominus}}\, d\xi_t-\int_0^{\sigma_\cI}e^{-rt}G_1(\XX^{x,\nu,\xi}_t)\,{\scriptstyle{\oplus}}\, d\nu_t\Big],\\
\label{eq:Psi2}\Psi_2(x;\nu,\xi):=\EE\Big[\int_0^{\sigma_\cI}e^{-rt}L_2(\XX^{x,\nu,\xi}_t)\,{\scriptstyle{\oplus}}\, d\nu_t-\int_0^{\sigma_\cI}e^{-rt}G_2(\XX^{x,\nu,\xi}_t)\,{\scriptstyle{\ominus}}\, d\xi_t\Big],
\end{align}
where $r>0$ is the discount rate and the integrals are defined below. 

To avoid dealing with controls producing infinite payoffs, we restrict our attention to the couples $(\nu,\xi)\in\cS\times\cS$ for which 
\begin{align}
\label{eq:Ic1}\EE&\Big[\int_0^{\sigma_\cI}e^{-rt}|L_1(\XX^{x,\nu,\xi}_t)|\,{\scriptstyle{\ominus}}\, d\xi_t+\int_0^{\sigma_\cI}e^{-rt}|G_1(\XX^{x,\nu,\xi}_t)|\,{\scriptstyle{\oplus}}\, d\nu_t\Big]<+\infty,\\
\label{eq:Ic2}\EE&\Big[\int_0^{\sigma_\cI}e^{-rt}|L_2(\XX^{x,\nu,\xi}_t)|\,{\scriptstyle{\oplus}}\, d\nu_t+\int_0^{\sigma_\cI}e^{-rt}|G_2(\XX^{x,\nu,\xi}_t)|\,{\scriptstyle{\ominus}}\, d\xi_t\Big]<+\infty.
\end{align}
We denote the space of such couples by $\cS^\circ\times\cS^\circ$.

A definition of the integrals with respect to the controls in presence of state dependent costs requires some attention because simultaneous jumps of $\xi$ and $\nu$ may be difficult to handle. An extended discussion on this matter is provided in Appendix \ref{app-setM}. Here we consider the class of admissible strategies (see Remark \ref{rem:strat} below)
\begin{align}
\label{def:admc}
\cM:=\{(\nu,\xi) \in \mathcal{S}^\circ \times \mathcal{S}^\circ\,:\,\PP_x(\Delta\nu_t\cdot\Delta\xi_t>0)=0\:\,\,\text{for all $t\ge 0$ and $x\in\cI$}\}.
\end{align}

Following \cite{Zhu92} (see also \cite{KZ15,LZ11} among others) we define the discounted costs of controls by
\begin{align}
\label{eq:int1}&\int_0^{T}e^{-rt}g(\XX^{x,\nu,\xi}_t)\,{\scriptstyle{\ominus}}\, d\xi_t=\int_0^{T}e^{-rt}g(\XX^{x,\nu,\xi}_t)d\xi^c_t+\sum_{t < T}e^{-rt} \int^{\Delta\xi_t}_0 g(\XX^{x,\nu,\xi}_t-z)dz\,,\\
\label{eq:int2}&\int_0^{T}e^{-rt}g(\XX^{x,\nu,\xi}_t)\,{\scriptstyle{\oplus}}\, d\nu_t=\int_0^{T}e^{-rt}g(\XX^{x,\nu,\xi}_t)d\nu^c_t+\sum_{t < T}e^{-rt} \int^{\Delta\nu_t}_0 g(\XX^{x,\nu,\xi}_t+z)dz\,,
\end{align}
for $T>0$, $(\nu,\xi)\in\cM$, and for any function $g$ such that the integrals are well defined.

Throughout the paper we take functions $G_i$ and $L_i$ satisfying
\begin{ass}\label{ass:GL}
$G_i,\,L_i:\overline{\cI}\to\RR\cup\{\pm\infty\}$, with $L_i<G_i$ on $\cI$ and with $G_i\in C^1(\cI)$ and $L_i\in C(\cI)$. Moreover the following asymptotic behaviours hold
$$\limsup_{x\to\underline{x}}\Big|\frac{G_i}{\phi}\Big|(x)=0\quad \text{and}\quad \limsup_{x\to\overline{x}}\Big|\frac{G_i}{\psi}\Big|(x)=0.$$
\end{ass}
Nash equilibria for the game are defined in the following way.
\begin{definition}\label{def:Nash-c}
For $x\in\cI$ we say that a couple $(\nu^*,\xi^*)\in\cM$ is a Nash equilibrium if and only if
\begin{align*}
\big|\Psi_i(x;\nu^*,\xi^*)\big|<+\infty,\quad i=1,2,
\end{align*}
and
\begin{align}
\left\{
\begin{array}{ll}
\Psi_1(x;\nu^*,\xi^*)\ge \Psi_1(x;\nu,\xi^*) & \text{for any $\nu\in\cS$ s.t.~$(\nu,\xi^*)\in\cM$},\\
\Psi_2(x;\nu^*,\xi^*)\ge \Psi_2(x;\nu^*,\xi) & \text{for any $\xi\in\cS$ s.t.~$(\nu^*,\xi)\in\cM$}.
\end{array}
\right.
\end{align}
We also say that $V_i(x):=\Psi_i(x;\nu^*,\xi^*)$ is the value of the game for the $i$-th player relative to the equilibrium.
\end{definition}

\begin{remark}
\label{rem:strat}
In several problems of interest for applications, the functionals \eqref{eq:Psi1} and \eqref{eq:Psi2} may be rewritten as the sum of three terms: an integral in time of a state dependent running profit, plus two integrals with respect to the controls, with constant instantaneous costs (see, e.g.,  \cite{DeAFe14}, \cite{GP05} and \cite{MZ07} for similar functionals in the case of single agent optimisation problems). In such cases, the condition in \eqref{def:admc} relative to jumps of the admissible strategies is not needed. In fact, we show in Appendix \ref{app-setM} that if at least one player picks a control that reflects the process at a fixed boundary (i.e.~solving one of the problems in Lemma \ref{lem:Sk01}), then the other player has no incentives in picking strategies outside of the class $\cM$. 
\end{remark}

\begin{remark}\label{rem:Ic}
It is worth noticing that, given $a,b\in\cI$ with $a<b$, the couple of controls $(\nu^a,\xi^b)$ which solves $\textbf{SP}(a,b;x)$ belongs to $\cM$. In fact, one can easily check that $(\nu^a,\xi^b)$ satisfy \eqref{eq:Ic1} and \eqref{eq:Ic2}, for example by looking at the proof of Lemma 2.1 in \cite{SLG84}. Moreover, by construction we have $\PP_x(\Delta\nu^a_t\cdot\Delta\xi^b_t>0)=0$ for all $t\ge0$. 
\end{remark}

\begin{remark}
Nash equilibria could in principle exist in broader sets than $\cM$. However this fact does not \emph{per se} add useful information. In fact, unless some additional optimality criterion is introduced (for example maximisation of the total profit of the two players), it is often impossible to rank multiple equilibria according to the players' individual preferences. In this paper we content ourselves with equilibria in $\cM$, as these lead to explicit solutions and to the desired connection between OS and SSC. 
\end{remark}

\subsection{The game of stopping}

In this section we introduce a 2-player nonzero-sum game of stopping where the underlying process is $X^x$ as in \eqref{state:XX}. This is the game which we show is linked to the game of controls introduced in the previous section. 

Denote by $\mathcal{T}$ the set of $\FF$-stopping times. The $i$-th player chooses $\tau_i\in\mathcal{T}$ with the aim of minimising an expected cost functional $\cJ_i(\tau_1,\tau_2;x)$, and the game ends at $\tau_1\wedge\tau_2$. This game has payoffs of immediate stopping given by the functions $G_i$ and $L_i$ appearing in the functionals \eqref{eq:Psi1} and \eqref{eq:Psi2} of the game of control. More precisely we set
\begin{align}
\label{functional0}&\mathcal{J}_1(\tau_1,\tau_2;x):=\EE\Big[e^{-\int^{\tau_1}_0(r-\mu'(X^x_s))ds}G_1(X^x_{\tau_1})\mathds{1}_{\{\tau_1 < \tau_2\}} + e^{-\int^{\tau_2}_0(r-\mu'(X^x_s))ds}L_1(X^x_{\tau_2})\mathds{1}_{\{\tau_1 \geq \tau_2\}}\Big],\\
\label{functional0b}&\mathcal{J}_2(\tau_1,\tau_2;x):=\EE\Big[e^{-\int^{\tau_2}_0(r-\mu'(X^x_s))ds}G_2(X^x_{\tau_2})\mathds{1}_{\{\tau_2 \le \tau_1\}} + e^{-\int^{\tau_1}_0(r-\mu'(X^x_s))ds}L_2(X^x_{\tau_1})\mathds{1}_{\{\tau_2 > \tau_1\}}\Big].
\end{align}

As in the case of the game of controls, also here we introduce the notion of Nash equilibrium.
\begin{definition}
\label{def:Nash}
For $x \in \cI$ we say that a couple $(\tau^*_1,\tau^*_2) \in \mathcal{T} \times \mathcal{T}$ is a Nash equilibrium 
if and only if
\begin{align*}
\big|\cJ_i(\tau^*_1,\tau^*_2;x)\big|<+\infty,\quad i=1,2
\end{align*}
and
\begin{equation}
\label{Nashequilibrium}
\left\{
\begin{array}{ll}
\mathcal{J}_1(\tau_1^*,\tau_2^*;x) \leq \mathcal{J}_1(\tau_1,\tau_2^*;x), \quad \forall\,\tau_1 \in \mathcal{T}, \\[+5pt]
\mathcal{J}_2(\tau_1^*,\tau_2^*;x) \leq \mathcal{J}_2(\tau_1^*,\tau_2;x), \quad \forall\,\tau_2 \in \mathcal{T}.
\end{array}
\right.
\end{equation}
We also say that $v_i(x):=\mathcal{J}_i(\tau_1^*,\tau_2^*;x)$ is the value of the game for the $i$-th player relative to the equilibrium.
\end{definition}

Our choice for the game of stopping is motivated by an heuristic argument which is well known in the economic literature on irreversible (partially reversible) investment problems. We briefly illustrate the main ideas below.

In our game of controls, both players are faced with the question of \emph{how} to use their control in order to maximise an expected payoff. This might be interpreted as the problem of two investors who must decide \emph{how to invest} a unit of capital in order to maximise their future expected profits. In mathematical economics literature (see, e.g., \cite{DP}) the question is known to be equivalent to the one of \emph{timing} the investment of one unit of capital. The equivalence can be formally explained via an analysis of marginal costs and benefits for each investor.

Here we take the point of view of player $1$, but symmetric arguments can be applied to player $2$. Given an investment strategy $\nu$, player $1$ pays a marginal cost equal to $G_1$ per unit of investment. However, the upward shift in the controlled dynamics (due to $\nu$) modifies the current level of the state variable, and therefore also the player's expected future profit. Such a change in the expected future payoffs, per unit of invested capital, represents the marginal benefit for player $1$. As long as the marginal benefit is smaller than the marginal cost, then player $1$ should \emph{wait} and do nothing. On the contrary, at times when the marginal benefit equals or exceeds the marginal cost, it is clear that player $1$ should invest (at the optimum the marginal benefit is never strictly larger than the marginal cost). In this sense, player 1 is \emph{timing} the decision to incur a (marginal) cost $G_1$, in exchange for expected future profits. This explains the (random) payoff $G_1(X_{\tau_1})$ in \eqref{functional0}--\eqref{functional0b}, while the indicator $\mathds{1}_{\{\tau_1 < \tau_2\}}$ is due to the fact that the previous argument holds until the second player decides to invest. In particular, while player $1$ waits for her optimal time $\tau_1$ to invest, it may happen that player $2$ decides to invest first. This situation produces a marginal cost for player $1$ equal to $L_1$ (which here may be negative or positive), and explains the role of the (random) payoff $L_1(X_{\tau_2})\mathds{1}_{\{\tau_1 \ge \tau_2\}}$ in \eqref{functional0}--\eqref{functional0b}.  

Since investors try to minimise costs, we are naturally led to consider minimisation of the players' expected discounted marginal costs \eqref{functional0}--\eqref{functional0b}. The specific discount factor adopted here is due to the nature of the underlying controlled diffusion, and it is a technical point which will become clear in the analysis below.

\section{The main result}
\label{sec:main}

Here we prove the key result of the paper (Theorem \ref{thm:labomba}), i.e.~a differential link between the value functions $v_i$, $i=1,2$ relative to Nash equilibria in the game of stopping and the value functions $V_i$, $i=1,2$ relative to Nash equilibria in the game of control. 
The result holds when the equilibrium stopping times for $X$ are hitting times to suitable thresholds so that the related optimally controlled $\XX$ is reflected at such thresholds. 

Theorem \ref{thm:labomba} relies on assumptions regarding the existence of a Nash equilibrium in the game of stopping and suitable properties of the associated values $v_1$ and $v_2$. It was shown in \cite{DeAFeMo15} that such requirements hold in a broad class of examples, and we will summarise results of \cite{DeAFeMo15} in Proposition \ref{prop:summ} below, for completeness.

For a given connected set $\mathcal{O}\subseteq\cI$, in the theorem below we will make use of the Sobolev space $W^{2,\infty}_{loc}(\mathcal{O})$. This is the space of functions which are twice differentiable in the weak sense on $\mathcal{O}$, and whose weak derivatives up to order two are functions in $L^\infty_{loc}(\mathcal{O})$. We will also use that if $u\in W^{2,\infty}_{loc}(\mathcal{O})$, then $u\in C^1(\mathcal{O})$ by Sobolev embedding \cite[Ch.~9, Cor.~9.15]{Br}.

\begin{theorem}\label{thm:labomba}
Suppose there exist $a_*,b_*$ with $\underline{x}<a_*<b_*<\overline{x}$ such that the following conditions hold: 
\begin{itemize}
\item[(a)] The stopping times
\begin{align}
\tau^*_1:=\inf\{t>0\,:\,X^x_t\leq a_*\},\qquad \tau^*_2:=\inf\{t>0\,:\,X^x_t\geq b_*\}
\end{align}
form a Nash equilibrium for the game of stopping as in Definition \ref{def:Nash};
\item[(b)] The value functions $v_i(x):=\cJ_i(\tau^*_1,\tau^*_2;x)$, $i=1,2$ are such that $v_i\in C(\cI)$, $i=1,2$ with $v_1\in W^{2,\infty}_{loc}(\underline{x},b_*)$ and $v_2\in W^{2,\infty}_{loc}(a_*,\overline{x})$ ; 
\item[(c)] $v_1=G_1$ in $(\underline{x},a_*]$, $v_1=L_1$ in $[b_*,\overline{x})$ and $v_2=G_2$ in $[b_*,\overline{x})$, $v_2=L_2$ in $(\underline{x},a_*]$. Moreover they solve the boundary value problem
\begin{align}
\label{bv1}&\big(\cL_Xv_i-(r-\mu')v_i\big)(x) =0, & a_*<x<b_*, \:i=1,2\\[+3pt]
\label{bv2}&\big(\cL_Xv_1-(r-\mu')v_1\big)(x) \ge 0, & \underline{x}<x\le a_*\\[+3pt]
\label{bv3}&\big(\cL_Xv_2-(r-\mu')v_2\big)(x) \ge 0, & b_*\le x<\overline{x}\\[+3pt]
\label{bv4}&v_i\le G_i,& x\in\cI,\:i=1,2.
\end{align}
\end{itemize} 

Then, the strategy profile that prescribes to reflect $\widetilde{X}$ at the two barriers $a^*$ and $b^*$ (up to a possible initial jump) forms a Nash equilibrium for the game of control (cf.\ Definition \ref{def:Nash-c}). In particular, for $x \in \cI$ and $t\geq 0$, such an equilibrium is realised by the couple of controls 
\begin{align}\label{optC}
\nu^*_t := \mathds{1}_{\{t>0\}}\big[(a_*-x)^+ + \nu^{a_*}_t\big], \qquad \xi^*_t := \mathds{1}_{\{t>0\}}\big[(x-b_*)^+ + \xi^{b_*}_t\big],
\end{align}
where $(\nu^{a_*},\xi^{b_*})$ uniquely solves Problem $\textbf{SP}(a_*,b_*; (x \vee a_*)\wedge b_*)$. Finally, the value functions $V_i(x)=\Psi_i(x;\nu^{*},\xi^{*})$, $i=1,2$ are given by
\begin{align}
V_1(x) & = \kappa_1 +\int^x_{a_*}v_1(z)dz, \qquad x\in\cI,\\
V_2(x) & = \kappa_2 +\int^{b_*}_{x}v_2(z)dz, \qquad x\in\cI,
\end{align}
with
\begin{align}\label{c1c2}
\kappa_1:=\frac{1}{r}\Big(\frac{\sigma^2}{2}\, G'_1+\mu\, G_1\Big)(a_*),\qquad \kappa_2:=-\frac{1}{r}\Big(\frac{\sigma^2}{2}\, G'_2+\mu\, G_2\Big)(b_*).
\end{align}  
\end{theorem} 
\begin{proof}
The proof is by direct check and it is performed in two steps.
\vspace{+4pt}

\emph{Step $1$.} The functions
\begin{align}
\label{u1} u_1(x) & = \kappa_1 +\int^x_{a_*}v_1(z)dz, \qquad x\in\cI,\\
\label{u2} u_2(x) & = \kappa_2 +\int^{b_*}_{x}v_2(z)dz, \qquad x\in\cI,
\end{align}
with $\kappa_1$ and $\kappa_2$ as in \eqref{c1c2}, are $C^1$ on $\cI$ (by continuity of $G_i$ and $L_i$ on ${\cI}$) with $u_1\in C^2(\underline{x},b_*)$ since $v_1\in C^1(\underline{x},b_*)$, and $u_2\in C^2(a_*,\overline{x})$ since $v_2\in C^1(a_*,\overline{x})$. We now show that $u_1$, $u_2$ and the boundaries $a_*$, $b_*$ solve the system of coupled variational problems
\begin{align}\label{HJB1}
\left\{
\begin{array}{ll}
(\cL_{\XX}u_1-ru_1)(x)=0, & x\in(a_*,b_*)\\[+4pt]
(\cL_{\XX}u_1-ru_1)(x)\le 0, & x\in(\underline{x},b_*)\\[+4pt]
u'_1(x)\le G_1(x), & x\in(\underline{x},b_*)\\[+4pt]
u'_1(x)=G_1(x), & x\in(\underline{x},a_*)\\[+4pt]
u'_1(x)=L_1(x), & x\in(b_*,\overline{x})
\end{array}
\right.
\end{align}
and 
\begin{align}\label{HJB2}
\left\{
\begin{array}{ll}
(\cL_{\XX}u_2-ru_2)(x)=0, & x\in(a_*,b_*)\\[+4pt]
(\cL_{\XX}u_2-ru_2)(x)\le 0, & x\in(a_*,\overline{x})\\[+4pt]
u'_2(x)\ge -G_2(x), & x\in(a_*,\overline{x})\\[+4pt]
u'_2(x)=-G_2(x), & x\in(b_*,\overline{x})\\[+4pt]
u'_2(x)=-L_2(x), & x\in(\underline{x},a_*).
\end{array}
\right.
\end{align}

We will only give details about the derivation of \eqref{HJB2} as the ones for \eqref{HJB1} are analogous. The last three properties in \eqref{HJB2} follow by observing that $u'_2=-v_2$, and by using $v_2=G_2$ in $[b_*,\overline{x})$, $v_2=L_2$ in $(\underline{x},a_*]$, and \eqref{bv4} (cf.~$(c)$ in the statement of the theorem). To prove the first equation in \eqref{HJB2} we use the definition of $u_2$ (see \eqref{u2}) and explicit calculations to get
\begin{align}\label{eq:ab}
(\cL_{\XX}u_2-ru_2)(x)=-\frac{\sigma^2(x)}{2}v_2'(x)-\mu(x)v_2(x)-r\kappa_2-\int_{x}^{b_*} r v_2(z)dz.
\end{align}
Then we also use \eqref{bv1} to obtain that, for $x\in(a_*,b_*)$, 
\begin{align}\label{eq:r}
\int_x^{b_*} r v_2(z)dz=\int_x^{b_*}\big(\cL_Xv_2(z)+\mu'(z)v_2(z)\big)dz.
\end{align}  
Integrating by parts the right hand-side of \eqref{eq:r}, using $v_2(b_*)=G_2(b_*)$ and $v'_2(b_*)=G'_2(b_*)$, and substituting the result back into \eqref{eq:ab}, the right-hand side of \eqref{eq:ab} equals zero upon recalling the definition of $\kappa_2$ (see \eqref{c1c2}). Finally, to prove the second line in \eqref{HJB2} it is enough to notice that for $x\in[b_*,\overline{x})$ 
\begin{align}
\int_x^{b_*} r v_2(z)dz\ge \int_x^{b_*}\big(\cL_{X}v_2(z)+\mu'(z)v_2(z)\big)dz,
\end{align}  
by \eqref{bv3} and then argue as before.
\vspace{+4pt}

\emph{Step $2$.} We now proceed to a verification argument to show that $u_i=V_i$, $i=1,2$, and that the strategy profile \eqref{optC} forms a Nash equilibrium. We provide again full details only for $u_2$ as the proof follows in the same way for $u_1$. 

Recall the dynamics for $\XX^{\nu,\xi}$ from \eqref{state:X}, and notice that by definition \eqref{optC}, the couple of controls $(\nu^*,\xi^*)$ solves the Skorokhod reflection problem in $[a_*,b_*]$, up to an initial jump. Moreover, Remark \ref{rem:Ic} guarantees that $(\nu^*,\xi^*)\in\cM$.

First we show that $u_2\ge \Psi_2(x;\nu^*,\xi)$ for any admissible $\xi$. Take $\xi \in \mathcal{S}^\circ$ such that $(\nu^*,\xi) \in \mathcal{M}$. It is important to notice that $\nu^*$ in \eqref{optC} involves the control $\nu^{a^*}$ that solves $\textbf{SP}^{\xi}_{a^*+}(x\vee a_*)$ of Lemma \ref{lem:Sk01}, for an arbitrary $\xi$. Recalling that that $u_2\in C^2(a_*,\overline{x})$, we can apply It\^o-Meyer's formula, up to a localising sequence of stopping times, to the process $u_2(\XX^{x,\nu^*,\xi})$ (in particular we use that $\PP_x(\Delta\nu^*_t\cdot\Delta\xi_t>0)=0$ for all $t\ge0$). The integral with respect to the continuous part of the bounded variation process $\nu^*-\xi$ is the difference of the integrals with respect to $d\nu^{*,c}$ and $d\xi^{c}$. For $x\in\cI$ we obtain
\begin{align}\label{eq:ver1}
u_2(x)=&e^{-r \theta_y}u_2(\XX^{x,\nu^*,\xi}_{\theta_y})-\int_0^{\theta_y}e^{-rs}(\cL_{\XX}-r)u_2(\XX^{x,\nu^*,\xi}_s)ds-M_{\theta_y}\nonumber\\
&-\int_0^{\theta_y}e^{-rs}u'_2(\XX^{x,\nu^*,\xi}_s)d\nu^{*,c}_s+\int_0^{\theta_y}e^{-rs}u'_2(\XX^{x,\nu^*,\xi}_s)d\xi^{c}_s\\
&-\sum_{s<\theta_y}e^{-rs}\big(u_2(\XX^{x,\nu^*,\xi}_{s+})-u_2(\XX^{x,\nu^*,\xi}_{s})\big),\nonumber
\end{align}
where $M$ is 
\begin{align}
M_t:=\int_0^te^{-rs}\sigma(\XX^{x,\nu^*,\xi}_s)u'_2(\XX^{x,\nu^*,\xi}_s)d\WW_s,
\end{align}
and $\theta_y$ is the stopping time
\begin{align}
\label{stop:local}
\theta_y:=\inf\{u>0\,:\,\XX^{x,\nu^*,0}_{u} \ge y\},\quad \text{for}\,y > b_*.
\end{align}
Notice that for any $t \in (0,\theta_y]$ we have $a_* \leq \XX^{x,\nu^*,\xi}_{t} \leq \XX^{x,\nu^*,0}_{t} \leq y$, hence continuity of $\sigma$ and of $u'_2$ imply that $(M_t)_{t\leq\theta_y}$ is a martingale.

Since $(\nu^*,\xi)\in\cM$, the process $\XX^{x,\nu^*,\xi}$ is left-continuous and we have
\begin{align}\label{jump}
& \sum_{s<\theta_y}e^{-rs}\big(u_2(\XX^{x,\nu^*,\xi}_{s+})-u_2(\XX^{x,\nu^*,\xi}_{s})\big) \\
& = \sum_{s<\theta_y}e^{-rs}\big(u_2(\XX^{x,\nu^*,\xi}_{s+})-u_2(\XX^{x,\nu^*,\xi}_{s})\big)\big[\mathds{1}_{\{\Delta \nu^*_s > 0\}} + \mathds{1}_{\{\Delta \xi_s > 0\}}\big],\nonumber
\end{align}
where
\begin{align*}
& \sum_{s<\theta_y}e^{-rs}\big(u_2(\XX^{x,\nu^*,\xi}_{s+})-u_2(\XX^{x,\nu^*,\xi}_{s})\big)\mathds{1}_{\{\Delta \nu^*_s > 0\}} = \sum_{s<\theta_y}e^{-rs}\int_0^{\Delta\nu^*_s} u'_2(\XX^{x,\nu^*,\xi}_{s}+z)dz, \nonumber \\
& \sum_{s<\theta_y}e^{-rs}\big(u_2(\XX^{x,\nu^*,\xi}_{s+})-u_2(\XX^{x,\nu^*,\xi}_{s})\big)\mathds{1}_{\{\Delta \xi_s > 0\}} = - \sum_{s<\theta_y}e^{-rs}\int_0^{\Delta \xi_s} u'_2(\XX^{x,\nu^*,\xi}_{s}-z)dz.
\end{align*}
Hence \eqref{eq:ver1} may be written in a more compact form as (cf.\ \eqref{eq:int1}, \eqref{eq:int2})
\begin{align}\label{eq:ver2}
u_2(x)=&e^{-r \theta_y}u_2(\XX^{x,\nu^*,\xi}_{\theta_y})-\int_0^{\theta_y}e^{-rs}(\cL_{\XX}-r)u_2(\XX^{x,\nu^*,\xi}_s)ds-M_{\theta_y}\nonumber\\
&-\int_0^{\theta_y}e^{-rs}u'_2(\XX^{x,\nu^*,\xi}_s)\,{\scriptstyle\oplus}\,d\nu^{*}_s+\int_0^{\theta_y}e^{-rs}u'_2(\XX^{x,\nu^*,\xi}_s)\,{\scriptstyle\ominus}\,d\xi_s.
\end{align}
Now, we notice that the third and fifth formulae in \eqref{HJB2} imply that $u'_2\ge -G_2$ on $\cI$ and that $u'_2(\XX^{x,\nu^*,\xi}_s)=-L_2(\XX^{x,\nu^*,\xi}_s)$ for all $s$ in the support of $d\nu^*_s$ (i.e.~for all $s\ge 0$ s.t.~$\XX^{x,\nu^*,\xi}_s\le a_*$). Moreover, employing the second expression in \eqref{HJB2} jointly with the fact that $\XX^{x,\nu^*,\xi}_s\ge a_*$ for $s > 0$, we get 
\begin{align}\label{eq:ver3}
u_2(x)\ge&e^{-r \theta_y}u_2(\XX^{x,\nu^*,\xi}_{\theta_y})-M_{\theta_y}\nonumber\\
&+\int_0^{\theta_y}e^{-rs}L_2(\XX^{x,\nu^*,\xi}_s)\,{\scriptstyle\oplus}\,d\nu^{*}_s-\int_0^{\theta_y}e^{-rs}G_2(\XX^{x,\nu^*,\xi}_s)\,{\scriptstyle\ominus}\,d\xi_s.
\end{align}
By taking expectations we end up with
\begin{align}\label{eq:ver4}
u_2(x)\ge&\EE_x \Big[ e^{-r \theta_y}u_2(\XX^{\nu^*,\xi}_{\theta_y})+\int_0^{\theta_y}e^{-rs}L_2(\XX^{\nu^*,\xi}_s)\,{\scriptstyle\oplus}\,d\nu^{*}_s-\int_0^{\theta_y}e^{-rs}G_2(\XX^{\nu^*,\xi}_s)\,{\scriptstyle\ominus}\,d\xi_s\Big].
\end{align}

We aim at taking limits as $y\to\overline{x}$ in \eqref{eq:ver4}, and we preliminarily notice that $\theta_y\uparrow\sigma_{\mathcal{I}}$ as $y\to\overline{x}$, $\PP_x$-a.s.

\begin{itemize}
\item[ (i)] By \eqref{u2} it is easy to see that
\begin{align*}
\label{stima1}
|u_2(\XX^{\nu^*,\xi}_{\theta_y})| \leq  & \kappa_2 + \int_{a_*}^{b_*}|v_2(z)| dz + \int_{b_*}^{b_* \vee \XX^{\nu^*,\xi}_{\theta_y}}|G_2(z)| dz \nonumber \\
\leq & C_2 + \int_{b_*}^{b_* \vee \XX^{\nu^*,0}_{\theta_y}}|G_2(z)| dz \leq C_2 + \int_{b_*}^{y}|G_2(z)| dz, \nonumber
\end{align*}
for some $C_2 > 0$, and where we have used $v_2 = G_2$ on $[b_*, \overline{x})$ and $\XX^{\nu^*,\xi}_{\theta_y} \leq \XX^{\nu^*,0}_{\theta_y}\le y$ $\PP_x$-a.s. Hence we have
\begin{equation}
\label{stima2}
\EE_x \big[ e^{-r \theta_y}u_2(\XX^{\nu^*,\xi}_{\theta_y})\big] \geq -\EE_x\big[ e^{-r \theta_y}\big]\Big(C_2 + \int_{b_*}^{y}|G_2(z)| dz\Big).
\end{equation}
Using Assumption \ref{ass:boundary}, Lemma \ref{lem:refl} in appendix guarantees
\begin{equation}
\label{stima3}
\limsup_{y \uparrow \overline{x}} \EE_x\big[ e^{-r \theta_y}\big] \Big(C_2 + \int_{b_*}^{y}|G_2(z)| dz\Big) \leq 0,
\end{equation}
so that \eqref{stima2} yields
\begin{equation}
\label{stima4}
\liminf_{y \uparrow \overline{x}}\EE_x \big[ e^{-r \theta_y}u_2(\XX^{\nu^*,\xi}_{\theta_y})\big] \geq 0.
\end{equation}

\item[(ii)] 
Recall the integrability conditions \eqref{eq:Ic1} and \eqref{eq:Ic2} in the definition of $\cM$. Then, using that $\theta_y\uparrow\sigma_\cI$ as $y\uparrow\infty$, and applying the dominated convergence theorem, we obtain 
\begin{align*}
& \lim_{y\to\overline{x}}\EE_x\Big[\int_0^{\theta_y}e^{-rs}L_2(\XX^{\nu^*,\xi}_s)\,{\scriptstyle\oplus}\,d\nu^{*}_s-\int_0^{\theta_y}e^{-rs}G_2(\XX^{\nu^*,\xi}_s)\,{\scriptstyle\ominus}\,d\xi_s\Big] \nonumber \\
& = \EE_x\Big[\int_0^{\sigma_{\mathcal{I}}}e^{-rs}L_2(\XX^{\nu^*,\xi}_s)\,{\scriptstyle\oplus}\,d\nu^{*}_s-\int_0^{\sigma_{\mathcal{I}}}e^{-rs}G_2(\XX^{\nu^*,\xi}_s)\,{\scriptstyle\ominus}\,d\xi_s\Big].
\end{align*}
\end{itemize}

Finally, we combine items (i) and (ii) and take limits in \eqref{eq:ver4} as $y\to\overline{x}$ to get 
\begin{align}\label{eq:ver5}
u_2(x)\ge&\EE_x \Big[\int_0^{\sigma_{\mathcal{I}}}e^{-rs}L_2(\XX^{\nu^*,\xi}_s)\,{\scriptstyle\oplus}\, d\nu^{*}_s-\int_0^{\sigma_{\mathcal{I}}}e^{-rs}
G_2(\XX^{\nu^*,\xi}_s)\,{\scriptstyle\ominus}\, d\xi_s\Big].
\end{align}
Hence $u_2(x)\ge \Psi_2(x;\nu^*,\xi)$ for any $\xi\in\cS$ such that $(\nu^*,\xi)\in\cM$.

Now, repeating the steps above with $\xi=\xi^*$, the inequalities in \eqref{eq:ver3} and \eqref{eq:ver4} become strict equalities due to the fact that $\XX^{x,\nu^*,\xi^*}_t\in[a_*,b_*]$ for all $t>0$ and $u'_2(\XX^{x,\nu^*,\xi^*}_t)=-G_2(\XX^{x,\nu^*,\xi^*}_t)$ on $\text{supp}\{d\xi^*_t\}$. 
Moreover the process $u_2(\XX^{\nu^*,\xi^*})$ is bounded, so that passing to the limit as $y\to\overline{x}$ gives
\[
\lim_{y\uparrow \overline x}\EE_x\left[e^{-r\theta_y}u_2(\XX^{\nu^*,\xi^*}_{\theta_y})\right]=0
\]
by dominated convergence and Assumption \ref{ass:boundary}. Hence $u_2(x)=\Psi(x;\nu^*,\xi^*)=V_2(x)$.
\end{proof}

\begin{remark}
From the game-theoretic point of view, Nash equilibria of Theorem \ref{thm:labomba} above are Markov perfect \cite{MT01} (also called Nash equilibria in closed-loop strategies), i.e.\ equilibria in which players' actions only depend on the ``payoff-relevant'' state variable $\widetilde{X}$. Our result provides a simple construction of closed-loop Nash equilibria for specific continuous time stochastic games of singular control. Since this problem is yet to be solved in game theory in its full generality (see the discussion in Section 2 of \cite{BP09} and in \cite{Steg10}), our work contributes to fill this gap.
\end{remark}


\subsection{On the assumptions of Theorem \ref{thm:labomba}.}
\label{ass:thmbomba}

In this section we give sufficient conditions under which $a_*$ and $b_*$ as in Theorem \ref{thm:labomba} exist.
Moreover, in Remark \ref{rem:bds} we provide algebraic equations for $a_*$ and $b_*$ which can be solved at least numerically.
Recall $\phi$ and $\psi$, i.e.~the fundamental decreasing and increasing solutions to \eqref{ODE}, and recall that $r>\mu'(x)$ for $x\in\overline{\cI}$ by Asssumption \ref{ass:rate}. 
We need the following set of functions:
\begin{definition}\label{def:sets}
Let $\cA$ be the class of real valued functions $H \in C^2(\cI)$ such that
\begin{align}
\label{lims}&\limsup_{x\to\underline{x}}\Big|\frac{H}{\phi}\Big|(x)=0,\:\:\:\:\limsup_{x\to\overline{x}}\Big|\frac{H}{\psi}\Big|(x)=0\\[+4pt]
\label{lims2}&\quad\text{and}\quad\EE_x\bigg[\int_0^{\sigma_\cI} e^{-\int_0^t(r-\mu'(X_s))ds} \big|h(X_t)\big|dt\bigg] < \infty
\end{align}
for all $x\in\cI$, and with $h(x):= (\cL_X H - (r-\mu')H)(x)$. We denote by $\cA_1$ (respectively $\cA_2$) the set of all $H \in \cA$ such that $h$ is strictly positive (resp.\ negative) on $(\underline{x},x_h)$ and strictly negative (resp.\ positive) on $(x_h,\overline{x})$, for some $x_h \in \cI$ with $\liminf_{x \to \underline x}h(x)>0$ (resp.\ $\limsup_{x\to\underline{x}}h(x)<0$) and $\limsup_{x \to \overline x}h(x)<0$ (resp.\ $\liminf_{x\to\overline{x}}h(x)>0$).
\end{definition}
We also need the following assumption, which will hold in the rest of this section.
\begin{ass}\label{ass:GL2}
For $i=1,2$, it holds $G_i\in\cA_i$ and
\begin{align}\label{ass:lim}
\limsup_{x\to\underline{x}}\Big|\frac{L_i}{\phi}\Big|(x)<+\infty\:\:\:\:\text{and}\:\:\:\:
\limsup_{x\to\overline{x}}\Big|\frac{L_i}{\psi}\Big|(x)<+\infty.
\end{align}
Moreover, letting $\hat{x}_1$ and $\hat{x}_2$ in $\cI$ be such that 
\begin{align}
\label{k1}\{x:(\cL_XG_1-(r-\mu')G_1)(x)>0\}=(\underline{x},\hat{x}_1),\\
\label{k2}\{x:(\cL_XG_2-(r-\mu')G_2)(x)>0\}=(\hat{x}_2,\overline{x}),
\end{align}
we assume $\hat{x}_1<\hat{x}_2$.
\end{ass}
The above condition $\hat{x}_1<\hat{x}_2$ implies that, for any value of the process $X$, at least one player has a running benefit from waiting (see the introduction of \cite{DeAFeMo15}).

The proof of the next proposition is given in Appendix \ref{app-proofs}. In its statement we denote 
\begin{align}
\vartheta_i(x):=\frac{G_i'(x)\phi(x)-G_i(x)\phi'(x)}{w\,S'(x)},\qquad i=1,2,
\end{align}
with $w>0$ as in \eqref{Wronskian}. We also remark that the proposition holds under all the standing assumptions made so far in the paper (i.e.~Assumptions \ref{ass:coef}, \ref{ass:boundary}, \ref{ass:rate}, \ref{ass:GL} and \ref{ass:GL2}). For the reader's convenience we also recall that $\underline x$ and $\overline x$ are natural for $X$ if the process cannot start from $\underline x$ and $\overline x$ and, moreover, when started in $(\underline x,\overline x)$ cannot reach $\underline x$ or $\overline x$ in finite time. On the other hand, $\underline x$ is entrance-not-exit if the process can be started from $\underline x$, but if started from $x>\underline x$ it cannot reach $\underline x$ in finite time. We refer to pp.\ 14--15 in \cite{BS} for further details.

\begin{prop}
\label{prop:summ}
Each one of the conditions below is sufficient for the existence of $a_*$ and $b_*$ fulfilling $(a)$, $(b)$ and $(c)$ of Theorem \ref{thm:labomba}:
\begin{enumerate}
\item $\underline{x}$ and $\overline{x}$ are natural boundaries for $(X_t)_{t\ge 0}$.
\item $\underline{x}$ is an entrance-not-exit boundary and $\overline{x}$ is a natural boundary for $(X_t)_{t\ge 0}$; moreover the following hold
\begin{itemize}
\item[(2.i)] $\vartheta_1(\underline{x}+):=\lim_{x\downarrow\underline{x}}\vartheta_1(x)<(L_1/\psi)(x^\infty_2)$,
where $x^\infty_2$ uniquely solves $\vartheta_2(x)=(G_2/\psi)(x)$ in $(\hat{x}_2,\overline{x})$;
\item[(2.ii)] $\sup\{x>\underline{x}\,:\,L_1(x)=\vartheta_1(\underline{x}+)\psi(x)\}\le\hat{x}_2$;
\item[(2.iii)] $\lim_{x\uparrow\overline{x}}(L_1/\phi)(x)>-\infty$.
\end{itemize}
\end{enumerate}
\end{prop}

\begin{remark}
\label{rem:bds}
An important byproduct of our connection between nonzero-sum games of control and nonzero-sum games of stopping is that the equilibrium thresholds $a_*$ and $b_*$ of Theorem \ref{thm:labomba} are a solution to a system of algebraic equations which can be computed at least numerically. In the terminology of singular control theory, these equations correspond to the smooth-fit conditions $V''_1(a_*+)=G'_1(a_*)$ and $V''_2(b_*-)=-G'_2(b_*)$, and were obtained via a geometric constructive approach in \cite{DeAFeMo15} (see Theorem 3.2 therein). We recall the system here for completeness
\begin{align}\label{syst}
\left\{
\begin{array}{c}
\displaystyle \frac{G_1}{\phi}(a_*)-\frac{L_1}{\phi}(b_*)-\vartheta_1(a_*)\Big(\frac{\psi}{\phi}(a_*)-\frac{\psi}{\phi}(b_*)\Big)=0\,,\\[+10pt]
\displaystyle \frac{G_2}{\phi}(b_*)-\frac{L_2}{\phi}(a_*)-\vartheta_1(b_*)\Big(\frac{\psi}{\phi}(b_*)-\frac{\psi}{\phi}(a_*)\Big)=0\,,
\end{array}
\right.
\end{align}
where $a_*<\hat{x}_1$ and $b_*>\hat{x}_2$.

Uniqueness of the solution to \eqref{syst} is discussed in \cite[Thm.~3.8]{DeAFeMo15}.
\end{remark}


\section{A game of pollution control}
\label{sec:pollution}

In order to understand the nature of our Assumptions \ref{ass:GL} and \ref{ass:GL2}, and illustrate an application of our results, we present here a game version of a pollution control problem. 

A social planner wants to keep the level of pollution low while the productive sector of the economy (modeled as a single representative firm) wants to increase the production capacity. If we assume that the pollution level is proportional to the firm's production capacity (see for example \cite{JZ01,VdPdZ}), then the problem translates into a game of capacity expansion. Indeed, the representative firm aims at maximising profits by investing to increase the production level, whereas the social planner aims at keeping the pollution level under control through environmental regulations which effectively cap the maximum production rate.  

For the production capacity we consider a controlled geometric Brownian motion as in \cite{CH09,DeAFe14,GP05}, amongst others,
\begin{align}
d\XX^{\nu,\xi}_t=\hat \mu\XX^{\nu,\xi}_t dt+\hat \sigma \XX^{\nu,\xi}_t d\WW_t+d\nu_t-d\xi_t,\quad\XX^{\nu,\xi}_0=x\in\RR_+,
\end{align}
for some $\hat \mu\in\RR$ and $\hat \sigma>0$. The firm has running operating profit $\pi(x)$, which is $C^1$ and strictly concave, and a positive cost per unit of investment $\alpha_1(x)$. The social planner has an instantaneous utility function $u(x)$ which is $C^1$, decreasing and strictly concave\footnote{The social planner's utility decreases with increasing pollution levels. Moreover, if the pollution is high the marginal benefit from decreasing it is large, whereas if the pollution is low a further contraction of the economy has very little or no benefit.}.
Since imposing a reduction of production might also have some negative impact on social welfare (e.g., it might cause an increase in the level of unemployment), we introduce a positive `cost' (in terms of the expected total utility) associated to the social planner's policies and we denote it by $\alpha_2(x)$. For simplicity here we assume $\alpha_i(x)\equiv\alpha_i>0$, $i=1,2$, and the objective functionals for the firm, denoted by $\Psi_1$, and the social planner, denoted by $\Psi_2$, are given by
\begin{align}
\label{eq:Psi1b}&\Psi_1(x;\nu,\xi):=\EE_x\Big[\int^{\sigma_{\mathcal{I}}}_0e^{-r t}\pi(\XX^{\nu,\xi}_t)dt-\alpha_1\int_0^{\sigma_{\mathcal{I}}}e^{-rt} d\nu_t\Big],\\
\label{eq:Psi2b}&\Psi_2(x;\nu,\xi):=\EE_x\Big[\int^{\sigma_{\mathcal{I}}}_0e^{-r t}u(\XX^{\nu,\xi}_t)dt-\alpha_2\int_0^{\sigma_{\mathcal{I}}}e^{-rt} d\xi_t\Big].
\end{align}
Both players want to maximise their respective functional by picking admissible strategies from $\cM$.
As explained in Lemma \ref{lem:jumps1} below, in this context there is no loss of generality for our scopes in considering $\cM$ rather than $\cS^\circ\times\cS^\circ$. 

The game with functionals \eqref{eq:Psi1b}--\eqref{eq:Psi2b} will be tackled directly with the same methods developed in the previous sections. Indeed, the additional running cost terms require only a minor tweak to our method. Motivated by the analysis of the previous sections we look at the game of stopping where two players want to minimise the cost functionals below:
\begin{align}
\label{eq:J1}&\widehat{\cJ}_1(\tau_1,\tau_2;x):=\EE_x\left[e^{-(r-\hat \mu)\tau_1}\alpha_1\mathds{1}_{\{\tau_1<\tau_2\}}+\int_0^{\tau_1\wedge\tau_2}e^{-(r-\hat \mu)t}\pi'(X_t)dt\right],\\
\label{eq:J2}&\widehat{\cJ}_2(\tau_1,\tau_2;x):=\EE_x\left[e^{-(r-\hat \mu)\tau_2}\alpha_2\mathds{1}_{\{\tau_2\le\tau_1\}}-\int_0^{\tau_1\wedge\tau_2}e^{-(r-\hat \mu)t}u'(X_t)dt\right],
\end{align} 
where the underlying process solves
\[
dX_t=(\hat \mu+\hat \sigma^2)X_tdt+\hat \sigma X_t d W_t,\quad \text{for $t>0$}, \qquad X_0=x>0.
\]

Theorem \ref{thm:labomba} holds in this setting and links the game of control \eqref{eq:Psi1b}--\eqref{eq:Psi2b} to the game of stopping \eqref{eq:J1}--\eqref{eq:J2}. In particular, in the statement of Theorem \ref{thm:labomba} we should now refer to the games in \eqref{eq:Psi1b}--\eqref{eq:J2} and replace \eqref{bv1}--\eqref{bv4} by 
\begin{align}
\label{bv1b}&\big(\cL_Xv_1-(r-\hat \mu)v_1\big)(x) =-\pi'(x), & a_*<x<b_*, \:i=1,2\\[+3pt]
\label{bv1c}&\big(\cL_Xv_2-(r-\hat \mu)v_2\big)(x) =u'(x), & a_*<x<b_*, \:i=1,2\\[+3pt]
\label{bv2b}&\big(\cL_Xv_1-(r-\hat \mu)v_1\big)(x) \ge -\pi'(x), & \underline{x}<x\le a_*\\[+3pt]
\label{bv3b}&\big(\cL_Xv_2-(r-\hat \mu)v_2\big)(x) \ge u'(x), & b_*\le x<\overline{x}\\[+3pt]
\label{bv4b}&v_i(x)\le \alpha_i,& x\in\cI,\:i=1,2.
\end{align}
Moreover, the constants $\kappa_i$ are adjusted as follows 
\begin{align}
\kappa_1:=\frac{1}{r}\Big(\hat\mu\, \alpha_1+\pi\Big)(a_*),\qquad \kappa_2:=-\frac{1}{r}\Big(\hat\mu\, \alpha_2-u\Big)(b_*).
\end{align}  
Everything else remains the same, including the proof of the theorem, which can be repeated by following the exact same steps.

We would like now to discuss sufficient conditions under which the game of stopping \eqref{eq:J1}--\eqref{eq:J2} admits a Nash equilibrium. In order to refer directly to the results for the stopping game from Section \ref{ass:thmbomba} it is convenient to rewrite \eqref{eq:J1}--\eqref{eq:J2} in the form of \eqref{functional0}--\eqref{functional0b}.

Here $\cI=\RR_+$ because $X$ is a geometric Brownian motion. For $r>\hat \mu$ we define functions $\Pi'$ and $U'$ via the ODEs
\begin{align}\label{D1}
(\cL_{X}-(r-\hat\mu))\Pi'(x)=-\pi'(x),\qquad (\cL_{X}-(r-\hat\mu))U'(x)=u'(x),
\end{align}
and by imposing growth conditions at zero and infinity. In particular, letting $\tau_n:=\inf\{t\ge0\,:\,X_t\notin(\tfrac{1}{n},n)\}$ we require that
\begin{align}
\label{D2}
\lim_{n\to \infty}\EE_x\left[e^{-(r-\hat\mu)\tau_n}\Pi'(X_{\tau_n})\right]=\lim_{n\to \infty}\EE_x\left[e^{-(r-\hat\mu)\tau_n}U'(X_{\tau_n})\right]=0.
\end{align}
A specific choice for $\pi$ and $u$ is discussed below, and for now we observe that by Dynkin formula and \eqref{D1} we get
\begin{align}
&\EE_x\left[\int_0^{\tau\wedge\tau_n}e^{-(r-\hat\mu)t}\pi'(X_t)dt\right]=\Pi'(x)-\EE_x\left[e^{-(r-\hat\mu)(\tau\wedge\tau_n)}\Pi'(X_{\tau\wedge\tau_n})\right]\\
&\EE_x\left[\int_0^{\tau\wedge\tau_n}e^{-(r-\hat\mu)t}u'(X_t)dt\right]=\EE_x\left[e^{-(r-\hat\mu)(\tau\wedge\tau_n)}U'(X_{\tau\wedge\tau_n})\right]-U'(x).
\end{align}
Letting $n\to\infty$ in the above expressions, using \eqref{D2} and plugging the result back in \eqref{eq:J1}--\eqref{eq:J2} we obtain the original formulation for $\cJ_1$ and $\cJ_2$ (cf.\ \eqref{functional0}--\eqref{functional0b}) by setting 
\begin{align*}
&G_1(x)=\alpha_1-\Pi'(x), & G_2(x)=\alpha_2-U'(x),\\
&L_1(x)=-\Pi'(x), & L_2(x)=-U'(x).
\end{align*}
It only remains to verify that it is possible to choose $\pi$ and $u$ such that Assumption \ref{ass:GL2} and condition \eqref{D2} hold. Hence, we can apply Proposition \ref{prop:summ}.

We now set
\begin{align}
&\zeta_1(x):=(\cL_{X}G_1-(r-\hat\mu)G_1)(x)=\pi'(x)-(r-\hat\mu)\alpha_1\\
&\zeta_2(x):=(\cL_{X}G_2-(r-\hat\mu)G_2)(x)=-u'(x)-(r-\hat\mu)\alpha_2,
\end{align}
and we notice that $\zeta_1$ is decreasing by concavity of $\pi$ whereas $\zeta_2$ is increasing by concavity of $u$. 
For instance assuming Inada conditions 
\begin{align*}
& \lim_{x\to\infty}\pi'(x)=0, & \lim_{x\to0}\pi'(x)=+\infty,\\
& \lim_{x\to\infty}u'(x)=-\infty, & \lim_{x\to 0}u'(x)=0,
\end{align*}
we have that \eqref{k1} and \eqref{k2} hold for some $\hat{x}_i$, $i=1,2$, which depend on the specific choice of $\pi$ and $u$. 

Let us now consider the case of $\pi(x)=x^\lambda$ and $u(x)=-x^\delta$ where $\lambda\in(0,1)$ and $\delta>1$. For $r>\hat{\mu}$ and sufficiently large we can guarantee \eqref{lims2} and \eqref{D2}. Moreover, denoting by $\gamma_1$ (resp.\ $\gamma_2$) the positive (resp.\ negative) root of the second order equation $\tfrac{1}{2}\hat{\sigma}^2\gamma(\gamma-1)+(\hat{\mu} + \hat{\sigma}^2)\gamma-(r-\hat{\mu})=0$, conditions \eqref{lims} on $G_1$ and $G_2$ are satisfied if $\lambda>\max\{0,\gamma_2+1\}$ and $1<\delta<1+\gamma_1$. Clearly \eqref{ass:lim} holds by the same arguments.
Finally, we have 
\begin{align*}
\hat{x}_1=\Big(\frac{r\alpha_1}{\lambda}\Big)^{-\frac{1}{1-\lambda}},\qquad\hat{x}_2=\Big(\frac{r\alpha_2}{\delta}\Big)^{\frac{1}{\delta-1}},
\end{align*}
so that a suitable choice of $\alpha_1$ and $\alpha_2$ ensures that $\hat{x}_1<\hat{x}_2$.


\appendix

\section{Appendix}

\renewcommand{\theequation}{A-\arabic{equation}}

\subsection{\textbf{Cost integrals and the set of strategies $\cM$}}
\label{app-setM}

It is well known in the singular stochastic control literature that state dependent instantaneous costs of control give rise to questions concerning the definition of integrals representing the cumulative cost of exercising control. 

Zhu in \cite{Zhu92} provided a definition consistent with the classical verification argument used in SSC for the solution to an HJB equation derived by the Dynamic Programming Principle. This definition has been adopted in several other papers concerning explicit solutions of SSC problems (see \cite{KZ15,LZ11} among others), and this is also the one that we use in our \eqref{eq:int1} and \eqref{eq:int2}. 
Another, perhaps more natural, possibility is instead to define the integral as a Riemann-Stieltjes' integral as for example it was done by Alvarez in \cite{Al00}. 

Despite this formal difference, it is remarkable that the two definitions for the cost of exercising control lead essentially to the same optimal strategies for problems of monotone follower type. In particular, it is possible to obtain Zhu's integral from the Riemann-Stieltjes' one by taking the limit as $n\to\infty$ of a sequence of controls that, at a given time $t$, make $n$ instantaneous jumps of length $h/n$ for a fixed $h>0$. The optimality of this behaviour is illustrated for example by Alvarez in Corollary 1 of \cite{Al00}, and it is often referred to as ``chattering policy''. The inconvenience with this approach is that the control obtained in the limit is not admissible in our $\cS$, and therefore optimisers can only be obtained in a larger class. 

Zhu's integral has proved to work very well in problems with monotone controls (representing for instance irreversible investments) or with controls of bounded variation (representing for instance partially reversible investment policies). In particular, the latter are often chosen in such a way that the controller's decision to invest/disinvest reflects the minimal decomposition of the control process (cf.\ \cite{DeAFe14}, \cite{FP13} and \cite{GZ15}, among others). In other words, investment and disinvestment do not occur at the same time, and this assumption is often justified by conditions on the absence of arbitrage opportunities. 

Here instead we have agents who use their controls independently, and it is unclear why \emph{a priori} they should decide not to contrast each other's moves by acting simultaneously. To elaborate more on this point and understand our choice of the set $\cM$, it is convenient to look at particular cases of our problem. 

In some instances, it is interesting to include in our functionals \eqref{eq:Psi1} and \eqref{eq:Psi2} a state-dependent running cost $\pi_i$ and use constant marginal costs/rewards of control $\alpha_i,\beta_i$ (see our example in Section \ref{sec:pollution} or problems studied in \cite{DeAFe14}, \cite{GP05} or \cite{MZ07}). The corresponding functionals read as follows
\begin{align}
\label{eq:Psi1c}\widehat{\Psi}_1(x;\nu,\xi):=\EE\Big[\int_0^{\sigma_{\mathcal{I}}}e^{-rt}\pi_1(\XX^{x,\nu,\xi}_t)dt+\int_0^{\sigma_{\mathcal{I}}}e^{-rt}\alpha_1\,d\xi_t-\int_0^{\sigma_{\mathcal{I}}}e^{-rt}\beta_1\,d\nu_t\Big],\\
\label{eq:Psi2c}\widehat{\Psi}_2(x;\nu,\xi):=\EE\Big[\int_0^{\sigma_{\mathcal{I}}}e^{-rt}\pi_2(\XX^{x,\nu,\xi}_t)dt+\int_0^{\sigma_{\mathcal{I}}}e^{-rt}\alpha_2\,d\nu_t-\int_0^{\sigma_{\mathcal{I}}}e^{-rt}\beta_2\,d\xi_t\Big].
\end{align}
In these cases the integrals with respect to the controls are simply understood as a Riemann-Stieltjes' integrals. For $\alpha_i<\beta_i$ we prove that, if one of the two players opts for a control that reflects the process at a threshold, then the other player's best response avoids simultaneous jumps of the controls. The condition $\alpha_i<\beta_i$ is the analogue in this context of the absence of arbitrage in papers like \cite{DeAFe14}, \cite{GP05} and \cite{MZ07}. The result is illustrated in the next lemma.
\begin{lemma}
\label{lem:jumps1}
Consider the game with functionals \eqref{eq:Psi1c}--\eqref{eq:Psi2c}. Let $a,b \in \cI$, recall Lemma \ref{lem:Sk01} and assume $\alpha_i<\beta_i$, $i=1,2$. If player 1 (resp.~player 2) chooses 
\begin{align}
\label{eq:tilde}
\text{$\tilde{\nu}^a_t:=\mathds{1}_{\{t>0\}}\big[(a-x)^+ + \nu^a_t\big]$ (resp.~$\tilde{\xi}^b:=\mathds{1}_{\{t>0\}}\big[(x-b)^+ + \xi^b_t\big]$)} 
\end{align}
where $\nu^a$ solves $\textbf{SP}^{\,\xi}_{a+}(x \vee a)$ (resp.~$\xi^b$ solves $\textbf{SP}^{\,\nu}_{b-}(x \wedge b)$) then the best reply
$\hat{\xi}_a:=\textrm{argmax}\,\widehat{\Psi}_2(x;\tilde{\nu}^a,\xi)$ is such that $(\tilde{\nu}^a,\hat{\xi}_a)\in\cM$ (resp.~$(\hat{\nu}_b,\tilde{\xi}^b)\in\cM$ with $\hat{\nu}_b:=\textrm{argmax}\,\widehat{\Psi}_1(x;\nu,\tilde{\xi}^b)$).
\end{lemma} 
\begin{proof}
Let $x,a\in\cI$ and $\xi\in\cS^\circ$ (recall \eqref{eq:Ic1}--\eqref{eq:Ic2}) and consider $\nu^a$ solving $\textbf{SP}^{\,\xi}_{a+}(x \vee a)$. We want to perform a pathwise comparison of the cost functional for player 2 under two different controls. In particular, we fix $\omega\in\Omega$ and assume that there exists (a stopping time) $t_0=t_0(\omega)>0$ such that $\big(\Delta\tilde{\nu}^a_{t_0}\cdot\Delta\xi_{t_0}\big)(\omega)>0$. With no loss of generality we may assume that $\XX^{x,\tilde{\nu}^a,\xi}_{t_0}(\omega)>a$ and that the downward jump $\Delta\xi_{t_0}$ is trying to push the process below $a$, i.e.~
\begin{align}\label{eq:xinot}
\Delta\xi_{t_0}(\omega)>[\XX^{x,\tilde{\nu}^a,\xi}_{t_0}-a](\omega).
\end{align}
This push causes the immediate reaction of the control $\tilde{\nu}^a$ and therefore a simultaneous jump of the two controls. The case in which $\XX^{x,\tilde{\nu}^a,\xi}_{t_0}(\omega)\le a$ can be dealt with in the same way up to trivial changes.

We denote by $\xi^0$ a control in $\cS^\circ$ such that 
\begin{align*}
\xi^0_t(\omega)=\left\{\begin{array}{ll}
\xi_t(\omega)\,, & t\le t_0\\[+4pt]
\xi_t(\omega)-[\Delta\xi_{t_0}-(\XX^{x,\tilde{\nu}^a,\xi}_{t_0}-a)](\omega)\,, & t>t_0.
\end{array}
\right.
\end{align*}
In particular, $\xi^0(\omega)$ is the same as $\xi(\omega)$ but the jump size at $t_0(\omega)$ is reduced so that the process is not pushed below $a$. For $\nu^a$ solving $\textbf{SP}^{\xi^0}_{a+}(x \vee a)$ the jump at $t_0$ is not triggered. Therefore, $\XX^{x,\tilde{\nu}^a,\xi^0}_{t_0+}(\omega)=a$ due only to the downward push given by $\xi^0$. Now we observe that the (random) Borel measure $d\tilde{\nu}^a$, induced by $\tilde{\nu}^a$ in response to $\xi$, differs from the measure $d\tilde{\nu}^a$, induced by $\tilde{\nu}^a$ in response to $\xi^0$, only for a mass at $t_0$ (which is needed to compensate for the jump of $\xi$). Moreover, since $\nu^a$ solves $\textbf{SP}^\xi_{a+}(x \vee a)$ for any $\xi$, then $\XX^{x,\tilde{\nu}^a,\xi}_{t}(\omega)=\XX^{x,\tilde{\nu}^a,\xi^0}_{t}(\omega)$ for all $t>0$, since $\XX^{x,\tilde{\nu}^a,\xi}_{t_0+}(\omega)=\XX^{x,\tilde{\nu}^a,\xi^0}_{t_0+}(\omega)=a$ and nothing else has changed for $t\neq t_0$. 

It is now easy to see that the couple $(\tilde{\nu}^a,\xi)$ requires an additional cost for player 2 compared to the couple $(\tilde{\nu}^a,\xi^0)$ and therefore cannot be optimal. For the sake of clarity here we denote by $\nu^{a,\xi}$ the solution to $\textbf{SP}^{\xi}_{a+}(x \vee a)$ and by $\nu^{a,\xi^0}$ the solution to $\textbf{SP}^{\xi^0}_{a+}(x \vee a)$, and also we set $\tilde{\nu}^{a,\xi}$ and $\tilde{\nu}^{a,\xi^0}$ as in \eqref{eq:tilde}.

So we obtain
\begin{align*}
\int_0^{\sigma_{\mathcal{I}}}&e^{-rt}\pi_2(\XX^{x,\tilde{\nu}^{a,\xi},\xi}_t)dt+\int_0^{\sigma_{\mathcal{I}}}e^{-rt}\alpha_2\,d\tilde{\nu}^{a,\xi}_t-\int_0^{\sigma_{\mathcal{I}}}e^{-rt}\beta_2\,d\xi_t\\
=&\int_0^{\sigma_{\mathcal{I}}}e^{-rt}\pi_2(\XX^{x,\tilde{\nu}^{a,\xi^0},\xi^0}_t)dt+\int_0^{\sigma_{\mathcal{I}}}e^{-rt}\alpha_2\,d\tilde{\nu}^{a,\xi^0}_t-\int_0^{\sigma_{\mathcal{I}}}e^{-rt}\beta_2\,d\xi^0_t\\
&+e^{-rt_0}(\alpha_2-\beta_2)[\Delta\xi_{t_0}-(X^{x,\tilde{\nu}^{a,\xi},\xi}_{t_0}-a)],
\end{align*}
and the last term is negative as $\alpha_2<\beta_2$ and by \eqref{eq:xinot}. Since the above argument can be repeated for any simultaneous jump of $\tilde{\nu}^a$ and $\xi$, and any $\omega\in\Omega$, the proof is complete.
\end{proof}

The point of the above lemma is that if costs of control are constant then a simple condition for the absence of arbitrage opportunities implies that if one player picks a reflecting strategy then the other one will pick a control such that $(\nu,\xi)\in\cM$. Therefore, under such assumptions, the equilibria constructed in Theorem \ref{thm:labomba} are also equilibria in the larger class $\cS^\circ\times\cS^\circ$.

\subsection{\textbf{Auxiliary results}}
\label{app-proofs}

We recall here the fundamental solutions $\phi$ and $\psi$ of \eqref{ODE}, and recall also that $\underline{x}$ and $\overline{x}$ are unattainable for $X$ of \eqref{state:XX} and for the uncontrolled diffusion $\XX^{0,0}$ of \eqref{state:X} (cf.\ Assumption \ref{ass:boundary}).
\begin{lemma}\label{lem:refl}
Let $a_*\in\cI$ be arbitrary but fixed. Take 
$$\nu^*_t:=\mathds{1}_{\{t>0\}}\big[(a_*-x)^+ + \nu^{a_*}_t\big],$$
with $\nu^{a_*}$ solving the Skorokhod reflection problem $\textbf{SP}^{\,0}_{a_*+}(x \vee a_*)$ of Lemma \ref{lem:Sk01}. For $y \in (a_*, \overline{x})$, set $\theta_y:=\inf\{t>0\,:\,\widetilde{X}^{\nu_*,0}_t\ge y\}$ and 
\begin{align*}
q(x,y):=\EE_x\big[e^{-r\theta_y}\big],\qquad x\in \cI.
\end{align*}
Then for $i=1,2$ we have 
\begin{align}\label{as1}
\lim_{y\uparrow\overline{x}}q(x,y)\left(1+\int^y_{a_*}|G_i(z)|dz\right)=0.
\end{align} 

Similarly let $b_*\in\cI$ be arbitrary but fixed. Take 
$$\xi^*_t:=\mathds{1}_{\{t>0\}}\big[(x-b_*)^+ + \xi^{b_*}_t\big],$$
with $\xi^{b_*}$ solution to the Skorokhod reflection problem $\textbf{SP}^{\,0}_{b_*-}(x \wedge b_*)$ of Lemma \ref{lem:Sk01}. For $y\in (\underline{x}, b_*)$, set $\eta_y:=\inf\{t>0\,:\,\widetilde{X}^{0,\xi_*}_t\le y\}$ and 
\begin{align*}
p(x,y):=\EE_x\left[e^{-r\eta_y}\right],\qquad x\in \cI.
\end{align*}
Then for $i=1,2$ we have 
\begin{align}\label{as2}
\lim_{y\downarrow\underline{x}}p(x,y)\left(1+\int^y_{b_*}|G_i(z)|dz\right)=0.
\end{align} 
\end{lemma}
\begin{proof}
We provide a full proof only for the first claim as the one for the second claim follows by similar arguments. Existence of a solution to $\textbf{SP}^{\,0}_{a_*+}(x \vee a_*)$ is obtained in \cite[Thm.~4.1]{Ta79} for coefficients $\mu,\,\sigma$ in \eqref{state:X} which are uniformly Lipschitz continuous. The relaxation to locally Lipschitz continuous coefficients (Assumption \ref{ass:coef}) follows by standard arguments as the ones used in the proof of our Lemma \ref{lem:Sk01} below.

We notice that 
$$\XX^{x,\nu^*,0}_t = \XX^{x \vee a_*,\nu^{a_*},0}_t, \qquad t>0,$$
and therefore $\theta_y$ is equal to $\tilde{\theta}_y:=\inf\{t>0\,:\,\XX^{x \vee a_*,\nu^{a_*},0}_t\ge y\}$ and $q(x,y) = q(a_*,y)$ for $x \leq a_*$. Functionals involving $\tilde{\theta}_y$ have well known analytical properties, and from now on we will make no distinction between $\theta_y$ and $\tilde{\theta}_y$.

For $x\geq y$ one has $q(x,y)=1$, whereas it is shown in Lemma 2.1 and Corollary 2.2 of \cite{SLG84} that the function $q(\,\cdot\,,y)$ solves
\begin{align}\label{ODE-q}
(\cL_{\XX}-r)q(x,y)=0,\qquad x\in(a_*,y),
\end{align}
with boundary conditions
\begin{align*}
q(y-,y):=\lim_{x\uparrow y}q(x,y)=1,\qquad q_x(a_*+,y):=\lim_{x\downarrow a_*}q_x(x,y)=0.
\end{align*}
In particular we refer to the condition at $a_*$ as the reflecting boundary condition. 

Since $q(\,\cdot\,,y)$ solves \eqref{ODE-q} then it may be written as
\begin{align*}
q(x,y)=A(y)\widetilde{\psi}(x)+B(y)\widetilde{\phi}(x),\qquad x\in(a_*,y),
\end{align*}
where $\widetilde{\psi}$ and $\widetilde{\phi}$ denote the fundamental increasing and decreasing solutions, respectively, of $(\cL_{\XX}-r)u=0$ on $\cI$.
By imposing the reflecting boundary condition we get
\begin{align*}
B(y)=-A(y)\frac{\widetilde{\psi}'(a_*)}{\widetilde{\phi}'(a_*)},
\end{align*}
which plugged back into the expression for $q$ gives
\begin{align}\label{eq:q}
q(x,y)=A(y)\left(\widetilde{\psi}(x)-\frac{\widetilde{\psi}'(a_*)}{\widetilde{\phi}'(a_*)}\widetilde{\phi}(x)\right).
\end{align}
Now, imposing the boundary condition at $y$ we also obtain
\begin{align}\label{eq:A}
A(y)=\left(\widetilde{\psi}(y)-\frac{\widetilde{\psi}'(a_*)}{\widetilde{\phi}'(a_*)}\widetilde{\phi}(y)\right)^{-1}.
\end{align}
Notice that $-\widetilde{\psi}'(a_*)/\widetilde{\phi}'(a_*)>0$, thus implying $A(y),B(y)>0$ and $q(x,y)>0$, as expected. Since the sample paths of $\XX^{\nu^*,0}$ are continuous for all $t>0$ then $y\mapsto q(x,y)$ must be strictly decreasing. Hence
\begin{align*}
q_y(x,y)=A'(y)\left(\widetilde{\psi}(x)-\frac{\widetilde{\psi}'(a_*)}{\widetilde{\phi}'(a_*)}\widetilde{\phi}(x)\right)<0
\end{align*}
which implies $A'(y)<0$ since the term in brackets is positive. From \eqref{eq:A} and direct computation we get
\begin{align*}
A'(y)=-\frac{1}{\left(A(y)\right)^2}\left(\widetilde{\psi}'(y)-\frac{\widetilde{\psi}'(a_*)}{\widetilde{\phi}'(a_*)}\widetilde{\phi}'(y)\right)
\end{align*}
and $A'(y)<0$ implies
\begin{align}\label{kin}
\left(\widetilde{\psi}'(y)-\frac{\widetilde{\psi}'(a_*)}{\widetilde{\phi}'(a_*)}\widetilde{\phi}'(y)\right)>0.
\end{align}
The latter inequality is important to prove \eqref{as1}.

The assumed regularity of $\mu$ and $\sigma$ (see Assumption \ref{ass:coef}) implies that $\widetilde{\psi}'$ solves $\cL_X u(x)-(r-\mu'(x))u(x)=0$ in $\cI$ (cf.\ \eqref{ODE}), and it can therefore be written as a linear combination of the fundamental increasing and decreasing functions $\psi$ and $\phi$. That is, 
\begin{equation}
\label{combination}
\widetilde{\psi}'(x) = \alpha\psi(x) + \beta\phi(x),
\end{equation}
for some $\alpha,\beta \in \mathbb{R}$. Analogously, 
\begin{align}\label{combination2}
\widetilde{\phi}'(x) = \gamma\psi(x) + \delta\phi(x).
\end{align}
Moreover since $\widetilde{\psi}'>0$ and $\widetilde{\phi}'<0$ in $\cI$, and $\underline{x}$ and $\overline{x}$ are unattainable for $X$, then it must be $\alpha,\beta\ge0$ and $\gamma,\delta\le 0$ (because $\psi(x)/\phi(x)\to\infty$ as $x\to\overline x$ and $\psi(x)/\phi(x)\to 0$ as $x\to\underline x$). Noticing that $y>a_*$ was arbitrary, the inequality \eqref{kin} now reads
\begin{align}\label{kin2}
\left(\alpha-\gamma\frac{\widetilde{\psi}'(a_*)}{\widetilde{\phi}'(a_*)}\right)\psi(y)+\left(\beta-\delta\frac{\widetilde{\psi}'(a_*)}{\widetilde{\phi}'(a_*)}\right)\phi(y)>0, \qquad y>a_*.
\end{align}

We aim at showing that $\alpha>0$ and we can do it by considering separately two cases.
\vspace{+4pt}

\noindent\emph{Case 1}. Assume $\gamma<0$. Since the second term in \eqref{kin2} can be made arbitrarily small by letting $y\to\overline{x}$ then it must be $\alpha>\gamma\widetilde{\psi}'(a_*)/\widetilde{\phi}'(a_*)>0$.
\vspace{+4pt}

\noindent\emph{Case 2}. Assume $\gamma=0$. If $\alpha=0$ then the first term on the left-hand side of \eqref{kin2} is zero and by using \eqref{combination} and \eqref{combination2} we get from \eqref{kin2}
\begin{align*}
0<\left(\beta-\delta\frac{\widetilde{\psi}'(a_*)}{\widetilde{\phi}'(a_*)}\right)\phi(y)=\left(\beta-\delta\frac{\beta\phi(a_*)}{\delta\phi(a_*)}\right)\phi(y)=0,
\end{align*}
hence a contradiction. So it must be $\alpha>0$.
\vspace{+4pt}

Finally, for fixed $x\in (\underline{x},y)$, there exists a constant $C=C(a_*,x)>0$ such that \eqref{eq:q} and \eqref{eq:A} give
\begin{align}\label{last}
0\le q(x,y)\left(1+\int^y_{a_*}|G_i(z)|dz\right)\le\frac{C}{\widetilde{\psi}(y)}\left(1+\int^y_{a_*}|G_i(z)|dz\right). 
\end{align}
Now letting $y\to\overline{x}$ we have $\widetilde{\psi}(y)\to\infty$ as $\overline{x}$ is unattainable for $\XX^{0,0}$. We have two possibilities: 
\begin{itemize}
\item[(a)] $\int_{a_*}^{\overline{x}}|G_i(z)|dz<+\infty$ and therefore \eqref{as1} holds trivially from \eqref{last};
\item[(b)] $\int_{a_*}^{\overline{x}}|G_i(z)|dz=+\infty$ so that by using de l'H\^opital rule in \eqref{last}, \eqref{combination} and Assumption \ref{ass:GL} we get
\begin{align*}
\lim_{y\to\overline{x}}\frac{1}{\widetilde{\psi}(y)}\int^y_{a_*}|G_i(z)|dz=\lim_{y\to\overline{x}}\frac{|G_i(y)|}{\alpha\psi(y)}=0.
\end{align*}
\end{itemize}
\end{proof}

\begin{proof}[\textbf{Proof of Lemma \ref{lem:Sk01}}.]
We provide here a short proof of the existence of a unique solution to the Skorokhod reflection 
problem $\textbf{SP}^{\,\xi}_{a+}(x)$. 

Notice that the drift and diffusion coefficients in the dynamics \eqref{state:X} are locally Lipschitz-continuous due to our Assumption \ref{ass:coef}. So we first prove the result for Lipschitz coefficients, and then extend it to locally Lipschitz ones. Notice that here we are not assuming sublinear growth of $\mu$ and $\sigma$ but we rely on non attainability of $\underline{x}$ and $\overline{x}$ for the uncontrolled process $\XX^{0,0}$.
Existence of a unique solution to problem $\textbf{SP}^{\,\nu}_{b-}(x)$ can be shown by analogous arguments. For simplicity, from now on we just write $\textbf{SP}^{\,\xi}_{a+}$ and omit the dependence on $x$.
\vspace{+5pt}

\noindent\emph{Step 1 - Lipschitz coefficients}. Here we assume $\mu,\sigma\in\text{Lip}(\cI)$ with constant smaller than $L>0$. Let $a\in\cI$, $x \geq a$ and $\xi\in\cS$, and consider the sequence of processes defined recursively by $X^{[0]}_t = x$, $\nu^{[0]}_t=0$, and 
\begin{align}
\label{steps}
\left\{
\begin{array}{l}
\displaystyle X^{[k+1]}_t = x + \int_0^t \mu(X^{[k]}_u) du + \int_0^t \sigma(X^{[k]}_u) dW_u + \nu^{[k+1]}_t - \xi_t,\\[+9pt]
\displaystyle \nu^{[k+1]}_t = \sup_{0 \leq s \leq t}\big[ a - x -\int_0^s \mu(X^{[k]}_u) du - \int_0^s \sigma(X^{[k]}_u) dW_u + \xi_s\big],
\end{array}
\right.
\end{align}
for any $k \geq 0$ and $t \geq 0$. Notice that at any step the process $X^{[k+1]}$ is kept above the level $a$ by the process $\nu^{[k+1]}$ with minimal effort, i.e.\ according to a Skorokhod reflection at $a$. The Lipschitz-continuity of $\mu$ and $\sigma$  allows to obtain from \eqref{steps} the estimate
\begin{equation}
\label{estimate}
\EE_x\Big[\sup_{0 \leq s \leq t}\big|X^{[k+1]}_s - X^{[k]}_s\big|^2\Big] \leq C\EE_x\Big[\int_0^t\big|X^{[k]}_s - X^{[k-1]}_s\big|^2 ds\Big],
\end{equation}
for $k\geq 1$ and for some positive $C:=C(x,a,L)$.
Since for $k=0$ one has $\EE_x[\sup_{0 \leq s \leq t}|X^{[1]}_s - x|^2] \leq R t$ for some $R:=R(x,a,L) > 0$, then an induction argument together with \eqref{estimate} yield 
\begin{equation}
\label{estimate-2}
\EE_x\Big[\sup_{0 \leq s \leq t}\big|X^{[k+1]}_s - X^{[k]}_s\big|^2\Big] \leq \frac{({R}_0 t)^{k+1}}{(k+1)!}, \quad k\geq 0,
\end{equation}
for some other positive ${R}_0:={R}_0(x,a,L)$.
Analogously,
\begin{equation}
\label{estimate-3}
\EE_x\Big[\sup_{0 \leq s \leq t}\big|\nu^{[k+1]}_s - \nu^{[k]}_s\big|^2\Big] \leq \frac{({R}_1 t)^{k+1}}{(k+1)!}, \quad k\geq 0,
\end{equation}
with ${R}_1:={R}_1(x,a,L) > 0$.

Thanks to \eqref{estimate-2} and \eqref{estimate-3} we can now proceed with an argument often used in SDE theory for the proof of existence of strong solutions (see, e.g., the proof of \cite[Ch.~5, Thm.~2.9]{KS}). That is, we use Chebyshev inequality and Borel-Cantelli's lemma to find that $(X^{[k+1]},\nu^{[k+1]})_{k\geq 0}$ converges a.s., locally uniformly in time, as $k \uparrow \infty$. We denote this limit by $(\widetilde{X}^{\nu^a,\xi},\nu^a)$. By Lipschitz continuity of $\mu$ and $\sigma$ and the same arguments as above we also obtain that the sequences $(\int_0^t \mu(X^{[k]}_u) du)_{k\geq 0}$ and $(\int_0^t \sigma(X^{[k]}_u) dW_u)_{k\geq 0}$ converge a.s., locally uniformly in time. Then we have a.s.\ (up to a possible subsequence)
\begin{align*}
\nu^a_t=\lim_{k\uparrow \infty }\nu^{[k+1]}_t = & \lim_{k\uparrow \infty }\sup_{0 \leq s \leq t}\big[ a - x -\int_0^s \mu(X^{[k]}_u) du - \int_0^s \sigma(X^{[k]}_u) dW_u + \xi_s\big] \\
 = & \sup_{0 \leq s \leq t}\big[ a - x -\int_0^s \mu(\widetilde{X}^{\nu^a,\xi}_u) du - \int_0^s \sigma(\widetilde{X}^{\nu^a,\xi}_u) dW_u + \xi_s\big].
\end{align*}
It thus follows that $(\widetilde{X}^{\nu^a,\xi},\nu^a)$ solve $\textbf{SP}^{\,\xi}_{a+}$. Finally, uniqueness can be proved as, e.g., in the proof of \cite[Thm.~4.1]{Ta79}.
\vspace{+5pt}

\noindent\emph{Step 2 - locally Lipschitz coefficients}. Here we assume $\mu$ and $\sigma$ as in Assumption \ref{ass:coef}.
Let $x_n\uparrow\overline{x}$ and define 
\begin{align*}
\mu_n(x)=\mu(x)\mathds{1}_{\{x\le x_n\}}+\mu(x_n)\mathds{1}_{\{x>x_n\}},\quad\sigma_n(x)=\sigma(x)\mathds{1}_{\{x\le x_n\}}+\sigma(x_n)\mathds{1}_{\{x>x_n\}}.
\end{align*}
For each $n$ we denote by $\textbf{SP}^{\,\xi\,(n)}_{a+}$~the Skorokhod problem $\textbf{SP}^{\,\xi}_{a+}$ but for the dynamics
\begin{align*}
dX_t=\mu_n(X_t)dt+\sigma_n(X_t)dW_t+d\nu_t-d\xi_t
\end{align*}
rather than for \eqref{state:X}.

Since for each $n$ we have $\mu_n$ and $\sigma_n$ uniformly Lipschitz on $[a,\overline{x})$, then Step 1 guarantees that there exists a unique $(X^{(n)},\nu^{(n)})$ that solves $\textbf{SP}^{\,\xi\,(n)}_{a+}$.
We denote $\tau_n:=\inf\{t>0\,:\,X^{(n)}_t\ge x_n \}$ and for all $t\le \tau_n$ we have
\begin{align}
X^{(n)}_t =& x + \int_0^t \mu_n(X^{(n)}_u) du + \int_0^t \sigma_n(X^{(n)}_u) dW_u + \nu^{(n)}_t - \xi_t\nonumber\\
=&x + \int_0^t \mu(X^{(n)}_u) du + \int_0^t \sigma(X^{(n)}_u) dW_u + \nu^{(n)}_t - \xi_t\\
\nu^{(n)}_t =& \sup_{0 \leq s \leq t}\big[ a - x -\int_0^s \mu_n(X^{(n)}_u) du - \int_0^s \sigma_n(X^{(n)}_u) dW_u + \xi_s\big]\nonumber\\
=&\sup_{0 \leq s \leq t}\big[ a - x -\int_0^s \mu(X^{(n)}_u) du - \int_0^s \sigma(X^{(n)}_u) dW_u + \xi_s\big].
\end{align}
Since the coefficients above do not depend on $n$, by construction the process $(X^{(n)}_t,\nu^{(n)}_t)$ also solves $\textbf{SP}^{\,\xi}_{a+}$ for $t\le \tau_n$. Uniqueness of the solution for $\textbf{SP}^{\,\xi\,(n)}_{a+}$ implies that $(X^{(n)}_t,\nu^{(n)}_t)$ is also the solution to $\textbf{SP}^{\,\xi\,(m)}_{a+}$ for $t\le\tau_m$, for each $m\le n$, and therefore the unique solution to $\textbf{SP}^{\,\xi}_a$ up to the stopping time $\tau_n$.

Fix an arbitrary $T>0$. For all $\omega\in\{\tau_n>T\}$ and all $t\le T$ we can define $(\widetilde{X}^{\nu^a,\xi}_t,\nu^a_t):=(X^{(n)}_t,\nu^{(n)}_t)$ so that the couple $(\widetilde{X}^{\nu^a,\xi}_t,\nu^a_t)$ is the unique solution to $\textbf{SP}^{\,\xi}_{a+}$ for $t\le T$. It remains to show that 
$\lim_{n\to\infty}\PP(\tau_n>T)= 1$
so that we have constructed a unique solution to $\textbf{SP}^{\,\xi}_{a+}$ for a.e.~$\omega\in\Omega$ up to time $T$. 

Let us consider first the case $\xi\equiv0$. It follows from Lemma \ref{lem:refl} that $\EE_x[e^{-r\theta_{x_n}}]\to 0$ as $n\to\infty$ with $\theta_{x_n}=\inf\{t>0\,:\,\XX^{\nu^a,0}_t\ge x_n\}$, and therefore $\theta_{x_n}\to \infty$ $\PP_x$-a.s. Hence 
\begin{align*}
\lim_{n\to \infty}\PP(\tau_n>T)=\lim_{n\to\infty}\PP(\theta_{x_n}>T)=1,
\end{align*}
because $\tau_n=\theta_{x_n}$ $\PP$-a.s.
To conclude it is suffices to notice that $\widetilde{X}^{\nu^a,\xi}_t\le\widetilde{X}^{\nu^a,0}_t$, for all $t>0$, and arbitrary $\xi\in\cS$. Then $\overline{x}$ is unattainable for $\widetilde{X}^{\nu^a,\xi}$ as well. 
\end{proof}

\medskip

\begin{proof}[\textbf{Proof of Proposition \ref{prop:summ}}.]
The proofs are contained in \cite{DeAFeMo15} and here we provide precise references to the relevant results in each case. In particular one must notice that Appendix A.3 of \cite{DeAFeMo15} addresses the specific setting of the state dependent discount factor $r-\mu'(x)$ that appears in our stopping functional \eqref{functional0}--\eqref{functional0b}.
\vspace{+4pt}

\emph{$1$.} It follows from Theorem 3.2 (and Appendix A.3) of \cite{DeAFeMo15}. 
\vspace{+4pt}

\emph{$2$.} It follows from Proposition 3.12 (and Appendix A.3) of \cite{DeAFeMo15}. For the sake of completeness here we notice that to prove that $x^\infty_2$ uniquely solves $\vartheta_2(x)=(G_2/\psi)(x)$ in $(\hat{x}_2,\overline{x})$ it is useful to change variables. Defining $y=(\psi/\phi)(x)=:F(x)$, where $F$ is strictly increasing, and introducing $\hat{G}_2(y):=\big[(G_2/\phi)\circ F^{-1}\big](y)$, $y>0$, it follows from simple algebra (cf.~Appendix A.1 of \cite{DeAFeMo15}) that $\vartheta_2(x)=(G_2/\psi)(x)$ is equivalent to $\hat{G}'_2(y)y=\hat{G}_2(y)$. It is shown in \cite[Lem.~3.6]{DeAFeMo15} that the latter equation has a unique root $y^\infty_2$ in the interval $(\hat{y}_2,\infty)$, with $\hat{y}_2:=F(\hat{x}_2)$ and $\infty=\lim_{x\uparrow\overline{x}}F(x)$. Therefore $x^\infty_2=F^{-1}(y^\infty_2)$ solves the initial problem in $(\hat{x}_2,\overline{x})$.
\end{proof}

\bigskip

\textbf{Acknowledgments.} The first named author was partially supported by EPSRC grant EP/K00557X/1; financial support by the German Research Foundation (DFG) through the Collaborative Research Centre 1283 ``Taming uncertainty and profiting from randomness and low regularity in analysis, stochastics and their applications'' is gratefully acknowledged by the second author. 

We thank two anonymous Referees for their pertinent comments that helped to improve a previous version of this paper. 
We also thank Cristina Costantini and Paavo Salminen for discussions and references on reflected diffusions and the Skorokhod reflection problem, and Jan-Henrik Steg for the valuable comments on closed-loop strategies in games of singular controls.


\end{document}